%% file: Main.tex
\documentclass[reqno]{amsart}
\usepackage{etex}
\usepackage{chemarr}
\usepackage{hyperref}
\usepackage{cleveref}
\usepackage{comment}

\usepackage[margin=1.5in]{geometry}

\usepackage{latexsym, amscd, amsfonts, eucal, mathrsfs, amsmath, 
amssymb, mathabx, amsthm, mathtools, xr, stmaryrd, color, enumerate, tikz}
\usepackage{bbm}
\usepackage{appendix}
\usepackage{multicol}
\usepackage[all]{xy}
\usepackage[textsize=scriptsize]{todonotes}

\theoremstyle{definition}
\newtheorem{nul}{}[section]
\newtheorem{dfn}[nul]{Definition}

\newtheorem{rmk}[nul]{Remark}

\newtheorem{cnstr}[nul]{Construction}
\newtheorem{ntn}[nul]{Notation}
\newtheorem{exm}[nul]{Example}

\newtheorem{rec}[nul]{Recollection}

\newtheorem{wrn}[nul]{Warning}
\newtheorem{qst}{Question}
\newtheorem*{dfn*}{Definition}
\newtheorem*{axm*}{Axiom}
\newtheorem*{ntn*}{Notation}
\newtheorem*{exm*}{Example}
\newtheorem*{exr*}{Exercise}
\newtheorem*{int*}{Intuition}
\newtheorem*{qst*}{Question}
\newtheorem*{rmk*}{Remark}

\theoremstyle{plain}

\newtheorem{thm}[nul]{Theorem}

\newtheorem{prop}[nul]{Proposition}
\newtheorem{lem}[nul]{Lemma}

\newtheorem{cor}[nul]{Corollary}

\newtheorem{cnv}[nul]{Convention}
\newtheorem{cnst}[nul]{Construction}
\newtheorem*{thm*}{Theorem}
\newtheorem*{prop*}{Proposition}
\newtheorem*{cor*}{Corollary}
\newtheorem*{lem*}{Lemma}
\newtheorem*{cnj*}{Conjecture}

\theoremstyle{definition}

\DeclareMathOperator{\Tr}{\mathrm{Tr}}

\DeclareMathOperator{\im}{\mathrm{im}}
\DeclareMathOperator*{\colim}{\mathrm{colim}}

\DeclareMathOperator{\Hom}{\mathrm{Hom}} 

\DeclareMathOperator{\Sym}{\mathrm{Sym}}

\DeclareMathOperator{\CP}{\mathbb{CP}}
\DeclareMathOperator{\Z}{\mathbb{Z}}
\DeclareMathOperator{\N}{\mathbb{N}}

\DeclareMathOperator{\Q}{\mathbb{Q}}

\DeclareMathOperator{\G}{\mathbb{G}}
\DeclareMathOperator{\F}{\mathbb{F}}

\DeclareMathOperator{\BPn}{BP \langle n \rangle}

\def\spoke{\Yright}
\def\vot{v_1^{\mu_p}}

\def\Bar{\mathrm{Bar}}

\def\H{\mathrm{H}}

\def\Cmu{\mathbb{CP}^{\infty} _{\mu_p}}
\def\Cmun{\mathbb{CP}^{n} _{\mu_p}}

\def\rhobar{\overline{\rho}}

\def\NCp{\mathrm{N}^{C_p} _{e}}
\def\HZ{\mathbb{Z}}
\def\HZp{\mathbb{Z}_{(p)}}

\def\HFp{\mathbb{F}_p}
\def\HZu{\underline{\mathbb{Z}}}
\def\HZpu{\underline{\mathbb{Z}}_{(p)}}
\def\HFu{\underline{\mathbb{F}}_p}

\def\gfp{\Phi^{C_p}}
\def\wt{\widetilde}
\def\Nm{\mathrm{Nm}}
\def\triv{\mathbbm{1}}
\def\tmf{\mathrm{tmf}}
\def\eps{\varepsilon}

\def\GL{\mathrm{GL}}
\def\BP{\mathrm{BP}}
\def\MU{\mathrm{MU}}
\def\End{\mathrm{End}}
\def\SS{\mathbb{S}}
\def\OO{\mathcal{O}}
\def\DD{\mathbb{D}}
\def\Honda{\mathrm{Honda}}
\def\QQ{\mathbb{Q}}
\def\Tmf{\mathrm{Tmf}}

\def\R{\mathbb{R}}
\def\Zlt{\mathbb{Z}_{(3)}}

\DeclarePairedDelimiter\abs{\lvert}{\rvert}%
\makeatletter
\let\oldabs\abs
\def\abs{\@ifstar{\oldabs}{\oldabs*}}

\usepackage{tikz}
\usetikzlibrary{matrix,arrows,decorations}
\usepackage{tikz-cd}

\usepackage{adjustbox}

\makeatletter
    \def\subsection{\@startsection{subsection}{2}%
      \z@{.5\linespacing\@plus.7\linespacing}{.25\linespacing}%
      {\normalfont\bfseries\centering}}
\makeatother

\makeatletter
\def\l@subsection{\@tocline{2}{0pt}{2.5pc}{5pc}{}}
\makeatother

\begin{document}

\title{Odd primary analogs of Real orientations}
\author{Jeremy Hahn}
\address{Department of Mathematics, Massachusetts Institute of Technology, Cambridge, MA, USA}
\email{jhahn01@mit.edu}
\author{Andrew Senger}
\address{Department of Mathematics, Harvard University, Cambridge, MA, USA}
\email{senger@math.harvard.edu}
\author{Dylan Wilson}
\address{Department of Mathematics, Harvard University, Cambridge, MA, USA}
\email{dwilson@math.harvard.edu}

\begin{abstract}
We define, in $C_p$-equivariant homotopy theory for $p>2$, a notion of $\mu_p$-orientation analogous to a $C_2$-equivariant Real orientation.  The definition hinges on a $C_p$-space $\mathbb{CP}^{\infty}_{\mu_p}$, which we prove to be homologically even in a sense generalizing recent $C_2$-equivariant work on conjugation spaces.

  We prove that the height $p-1$ Morava $E$-theory is $\mu_p$-oriented and that $\mathrm{tmf}(2)$ is $\mu_3$-oriented.  We explain how a single equivariant map $v_1^{\mu_p}:S^{2\rho} \to \Sigma^{\infty} \mathbb{CP}^{\infty}_{\mu_p}$ completely generates the homotopy of $E_{p-1}$ and $\tmf(2)$, expressing a height-shifting phenomenon pervasive in equivariant chromatic homotopy theory.
\end{abstract}


\maketitle

\tableofcontents

\vbadness 5000


\section{Introduction} \label{sec:Introduction}
\input{Introduction.tex}

\section{Orientation theory} \label{sec:Orientation}
\input{OrientationsFinal.tex}

\section{Evenness} \label{sec:Purity}
\input{Purity.tex}

\section{The homological evenness of $\Cmu$} \label{sec:Cmu-Purity}
\input{CP_purity.tex}

\section{Examples of homotopical evenness} \label{sec:Hmtpy-Purity}
\input{Hom_purity_examples.tex}

\section{$\vot$ and a formula for its span} \label{sec:vot}
\input{Vot.tex}

\section{The span of $\vot$ in height $p-1$ theories} \label{sec:Image}
\input{Image_in_Ep-1.tex}

\bibliographystyle{alpha}
\bibliography{bibliography}

\end{document}

%% file: Introduction.tex
The complex conjugation action on $\mathbb{CP}^{\infty}$ gives rise to a $C_2$-equivariant space, $\mathbb{CP}^{\infty}_{\mathbb{R}}$, with fixed points $\mathbb{RP}^{\infty}$.
The subspace $\mathbb{CP}^1_{\mathbb{R}}$ is invariant and equivalent as a $C_2$-space
to $S^{\rho}$, the one-point compactification of the real regular representation
of $C_2$.
A $C_2$-equivariant ring spectrum $R$ is \emph{Real oriented} if it is equipped with a map
$$\Sigma^{\infty} \mathbb{CP}^{\infty}_{\mathbb{R}} \to \Sigma^{\rho}R$$
such that the restriction 
$$S^{\rho} = \Sigma^{\infty} \mathbb{CP}^{1}_{\mathbb{R}} \to \Sigma^{\infty} \mathbb{CP}^{\infty}_{\mathbb{R}} \to \Sigma^{\rho}R$$
is the $\Sigma^{\rho}$-suspension of the unit map $S^0 \to R$.  Such a Real orientation induces a homotopy ring map 
$$\mathrm{MU}_{\mathbb{R}} \to R,$$
with domain the spectrum of Real bordism \cite{ArakiI,HuKriz}.  These orientations have proved invaluable to the study of $2$-local chromatic homotopy theory, leading to an explosion of progress surrounding the Hill--Hopkins--Ravenel solution of the Kervaire invariant one Problem \cite{HHR, RealGorenstein, HillMeier, KLWLandweberFlat, EPic, C4slice, InvertibleEmod, RealTate, HurewiczImage, RealOrEn, LTModels, NormF2}.

The above papers solve problems, at the prime $p=2$, that admit clear but often unapproachable analogs for odd primes.  To give two examples, the $3$ primary Kervaire problem remains unresolved \cite{HHRodd}, and substantially less precise information is known about odd primary Hopkins--Miller $EO$-theories \cite[Conjecture 1.12]{HoodBhatt}.

To rectify affairs at $p>2$, the starting point must be to find a $C_p$-equivariant space playing the role of $\mathbb{CP}^{\infty}_{\mathbb{R}}$.  This paper began as an attempt of the first two authors to understand a space proposed by the third.

\begin{cnstr}[Wilson]
For any prime $p$, let $\mathbb{CP}^{\infty}_{\mu_p}$ denote the fiber of the $C_p$-equivariant multiplication map
$$\left(\mathbb{CP}^{\infty}\right)^{\times p} \to \mathbb{CP}^{\infty},$$
where the codomain has trivial $C_p$-action
and the domain has $C_p$-action cyclically permuting the terms.
In other words, a map of spaces $X \to \mathbb{CP}^{\infty}_{\mu_p}$ consists of the data of:
\begin{itemize}
\item A $p$-tuple of complex line bundles $(\mathcal{L}_1,\mathcal{L}_2,\cdots,\mathcal{L}_p)$ on $X$.
\item A trivialization of the tensor product $\mathcal{L}_1 \otimes \mathcal{L}_2 \otimes \cdots \otimes \mathcal{L}_p$.
\end{itemize}
The action on $\Cmu$ is given by
$$(\mathcal{L}_1,\mathcal{L}_2,\cdots,\mathcal{L}_p) \mapsto (\mathcal{L}_p,\mathcal{L}_1,\cdots,\mathcal{L}_{p-1}).$$
\end{cnstr}

\begin{rmk} 
There is an equivalence of $C_2$-spaces
$\mathbb{CP}^{\infty}_{\mu_2} \simeq \mathbb{CP}^{\infty}_{\mathbb{R}}.$
In general, the non-equivariant space underlying $\mathbb{CP}^{\infty}_{\mu_{p}}$ is equivalent to $\left(\mathbb{CP}^{\infty}\right)^{\times p-1}$.  
The fixed points $\left( \mathbb{CP}^{\infty}_{\mu_p} \right)^{C_p}$ are equivalent to the classifying space $BC_p$, as can be seen by applying the fixed points functor $\left(\--\right)^{C_p}$ to the defining fiber sequence for $\Cmu$.
The key point here is that the $C_p$-fixed points of $\left(\mathbb{CP}^{\infty}\right)^{\times p}$ consist of the diagonal copy of $\mathbb{CP}^{\infty}$, and $BC_p$ is the fiber of the $p$th tensor power map $\CP^{\infty} \to \CP^{\infty}$.
%
\end{rmk}


To formulate the notion of Real orientation, it is essential to understand the inclusion of the bottom cell
$$S^{\rho} = \mathbb{CP}^{1}_{\mathbb{R}} \to \mathbb{CP}^{\infty}_{\mathbb{R}}.$$
At an arbitrary prime, the analog of this bottom cell is described as follows:

\begin{ntn}
We let $S^{\Yright}$ denote the cofiber of the unique non-trivial map of pointed $C_p$-spaces from $\left(C_{p}\right)_+$ to $S^0$.
This is the \emph{spoke sphere}, and it is a wedge of $(p-1)$ copies of $S^1$
with action on reduced homology given by the augmentation ideal in the group ring
$\mathbb{Z}[C_p]$. 
We denote the suspension $\Sigma S^{\Yright}$ of the spoke sphere by either $S^{1+\Yright}$ or $\mathbb{CP}^{1}_{\mu_p}$, and \Cref{cpmup-as-em-space} provides a natural inclusion
$$S^{1+ \Yright}= \mathbb{CP}^{1}_{\mu_p} \to \mathbb{CP}^{\infty}_{\mu_p}.$$
We will often also use $S^{1+\Yright}$ to denote $\Sigma^{\infty} S^{1+\Yright}$, and $S^{-1-\Yright}$ to denote its Spanier-Whitehead dual.
\end{ntn}

With this bottom cell in hand, we propose the following generalization of Real orientation theory:

\begin{dfn} \label{dfn:mup-orient}
A $\mu_p$-orientation of a $C_p$-equivariant ring $R$ is a map of spectra
$$\Sigma^{\infty} \mathbb{CP}^{\infty}_{\mu_p} \to \Sigma^{1+\Yright} R$$
such that the composite
$$S^{1+\Yright} = \Sigma^{\infty} \mathbb{CP}^{1}_{\mu_p} \to \Sigma^{\infty} \mathbb{CP}^{\infty}_{\mu_p} \to \Sigma^{1+\Yright} R$$
is the $S^{1+\Yright}$-suspension of the unit map $S^0 \to R$.
\end{dfn}

\begin{rmk} Applying the geometric fixed point functor $\Phi^{C_p}$ 
to a $\mu_p$-orientation we learn that the non-equivariant spectrum
$\Phi^{C_p}R$ has $p=0$ in its homotopy
groups.
\end{rmk}

\begin{rmk} \label{cpmup-as-em-space}
Let $\underline{\mathbb{Z}} \coloneq \mathrm{H}\underline{\mathbb{Z}}$ denote the $C_p$-equivariant Eilenberg--MacLane spectrum associated to the constant Mackey functor.  Then there is an equivalence of $C_p$-equivariant spaces
\[\Omega^{\infty} \Sigma^{1+\Yright} \underline{\mathbb{Z}} \simeq \Cmu.\]
Indeed, suspending and rotating the defining cofiber sequence $(C_p)_+ \to S^0 \to S^{\Yright}$ gives rise to a cofiber sequence $S^{1+\spoke} \to (C_p)_+ \otimes S^2 \to S^2$. Tensoring with $\underline{\mathbb{Z}}$ and applying $\Omega^{\infty}$ yields the defining fiber sequence for $\Cmu$.

Under this identification, the natural inclusion $\mathbb{CP}^{1}_{\mu_p} \to \mathbb{CP}^{\infty}_{\mu_p}$ is simply adjoint to the $\Sigma^{1+\Yright}$-suspension of the unit map $S^0 \to \underline{\mathbb{Z}}$.  In particular, the identification $\mathbb{CP}^{\infty}_{\mu_p} \simeq \Omega^{\infty}(\Sigma^{1+\Yright}\underline{\mathbb{Z}})$ gives a
canonical $\mu_p$-orientation of $\underline{\mathbb{Z}}$.
In contrast, Bredon cohomology with coefficients in the Burnside Mackey functor
cannot be $\mu_p$-oriented, since $p$ is nonzero in the
geometric fixed points.
\end{rmk}

In this paper we explore the interaction between $\mu_p$-orientations and chromatic homotopy theory in the simplest possible case: chromatic height $p-1$.  Specifically, we study the following height $p-1$ $\mathbb{E}_\infty$-ring spectra:

\begin{ntn}
  We let $E_{p-1}$ denote the height $(p-1)$ Lubin--Tate theory associated to the Honda formal group law over $\mathbb{F}_{p^{p-1}}$, with $C_p$-action given by a choice of order $p$ element in the Morava stabilizer group.  At $p=3$, we let $\mathrm{tmf}(2)$ denote the $3$-localized connective ring of topological modular forms with full level $2$ structure \cite{Vesnatmf2}.  The ring $\mathrm{tmf}(2)$ naturally admits an action by $\Sigma_3 \cong \mathrm{SL}_2(\mathbb{F}_2)$, and we restrict along an inclusion $C_3 \subset \Sigma_3$ to view $\tmf(2)$ as a $C_3$-equivariant ring spectrum.

The underlying homotopy groups of these spectra are given respectively by
$$\pi^{e}_*(E_{p-1}) \cong \mathbb{W} (\F_{p^{p-1}})\llbracket u_1,u_2,\cdots,u_{p-2} \rrbracket [u^{\pm}], \,\, |u_i|=0, |u|=-2,\, \text{ and}$$
$$\pi^{e}_*(\tmf(2)) \cong \mathbb{Z}_{(3)}[\lambda_1,\lambda_2], \,\, |\lambda_i|=4.$$

We will review the $C_p$-actions on the homotopy groups in \Cref{sec:Hmtpy-Purity}.
\end{ntn}

\begin{thm} 
For all primes $p$, there exists a $\mu_p$-orientation of the $C_p$-equivariant Morava $E$-theory $E_{p-1}$.
\end{thm}

\begin{thm}
The ($3$-localized) $C_3$-equivariant ring $\mathrm{tmf}(2)$ of topological modular forms with full level $2$ structure admits a $\mu_3$-orientation.
\end{thm}

Our second main result concerns the fact that, while
$$\pi_* E_{p-1} \cong \mathbb{W}(\F_{p^{p-1}}) \llbracket u_1,u_2,\dots,u_{p-2} \rrbracket [u^{\pm}]$$
has $(p-1)$ distinct named generators, the conglomeration of them is generated under the $\mu_p$-orientation by a \emph{single} equivariant map $v_1^{\mu_p}$.

\begin{cnstr}
In \Cref{sec:vot}, we will construct a map of $C_p$-equivariant spectra 
$$v^{\mu_p}_1:S^{2\rho} \to \Sigma^{\infty} \mathbb{CP}^{\infty}_{\mu_p}.$$
This map should be viewed as canonical only up to some indeterminacy, just as the classical class $v_1$ is only well-defined modulo $p$.  As was pointed out to the authors by Mike Hill, one choice of this map is given by \emph{norming} a non-equivariant class in $\pi_2^{e} \Cmu$.
\end{cnstr}

\begin{cnstr} \label{cnstr:v1mup}
Suppose a $C_p$-equivariant ring $R$ is $\mu_p$-oriented via a map
$$\Sigma^{\infty} \mathbb{CP}^{\infty}_{\mu_p} \to \Sigma^{1+\Yright} R,$$
so that we may consider the composite
$$
\begin{tikzcd}
S^{2\rho} \arrow{r}{\vot} & \Sigma^{\infty} \mathbb{CP}^{\infty}_{\mu_p} \arrow{r} & \Sigma^{1+\Yright} R.
\end{tikzcd}
$$
Using the dualizability of $S^{1+\Yright}$, this composite is equivalent to the data of a map
$$S^{2\rho-1-\Yright} \to R.$$
The non-equivariant spectrum underlying $S^{2\rho-1-\Yright}$ is (non-canonically) equivalent to a direct sum of $p-1$ copies of $S^{2p-2}$.
In particular, by applying $\pi^{e}_{2p-2}$ to the map $S^{2\rho-1-\Yright} \to R,$ one obtains a map from a rank $p-1$ free $\mathbb{Z}_{(p)}$-module to $\pi^{e}_{2p-2} R$.
\end{cnstr}

\begin{dfn}
Given a $C_p$-equivariant ring $R$ with a $\mu_p$-orientation, the \emph{span of $\vot$} will refer to the subset of $\pi^{e}_{2p-2} R$ consisting of the image of the rank $p-1$ free $\mathbb{Z}_{(p)}$-module constructed above.
\end{dfn}

\begin{thm}
For any $\mu_3$-orientation of $\tmf(2)$, the span of $v_1^{\mu_3}$ in $\pi^{e}_4 \tmf(2)$ is all of $\pi^{e}_4 \tmf(2)$.
\end{thm}

\begin{thm}
For any $\mu_p$-orientation of the height $p-1$ Morava $E$-theory $E_{p-1}$, the span of $\vot$ inside $\pi^{e}_{2p-2} E_{p-1}$ maps surjectively onto $\pi^{e}_{2p-2} E_{p-1}/(p,\mathfrak{m}^2)$.
\end{thm}

\begin{rmk}
  The map $S^{2\rho-1-\Yright} \to R$ associated to a $\mu_p$-oriented $R$ has an interpretation that may be more familiar to readers acquainted with the Hopkins--Miller computation of the fixed points of $E_{p-1}$.
  Specifically, by definition there is a cofiber sequence
  \[S^{2\rho-2} \xrightarrow{\mathrm{tr}} (C_p)_+ \otimes S^{2\rho-2} \to S^{2\rho-1-\Yright},\]
  where $\mathrm{tr}$ is the transfer.
  It follows that the map $S^{2\rho-1-\Yright} \to R$ determines a traceless element in $\pi_{2p-2} ^e R$, and the existence of such a traceless element was a key tool in the computations of \cite{NaveEO}.
\end{rmk}

\subsection{Homological and homotopical evenness}

Non-equivariantly, complex orientation theory is intimately tied to the notion of evenness.  A fundamental observation is that, since $\mathbb{CP}^{\infty}$ has a cell decomposition with only even-dimensional cells, any ring $R$ with $\pi_{2*-1} R \cong 0$ must be complex orientable.

In $C_2$-equivariant homotopy theory, a ring $R$ is called \emph{even} if $\pi^{C_2}_{*\rho-1} R \cong \pi^{e}_{2*-1} R \cong 0$, and it is a basic fact that any even ring is Real orientable \cite[\S 3.1]{HillMeier}.

In $C_p$-equivariant homotopy theory, we propose the appropriate notion of evenness to be captured by the following definition, which we discuss in more detail in \Cref{sec:Purity}:

\begin{dfn}
    We say that a $C_p$-equivariant spectrum $E$ is \textit{homotopically even} if the following conditions hold for all $n \in \Z$:
    \begin{enumerate}
        \item $\pi^{e}_{2n-1}  E = 0$.
        \item $\pi^{C_p}_{2n\rho-1}E = 0$
        \item $\pi^{C_p}_{2n\rho-2-\spoke} E = 0$
    \end{enumerate}
\end{dfn}

\begin{rmk} In the presence of condition (1), condition (3) is
equivalent to the statement that the transfer is surjective
in degree $2n\rho-2$.  Conditions $(1)$ and $(2)$ constrain certain slices of $E$, as we spell out in \Cref{rmk:slice-comment}.
\end{rmk}

\begin{rmk}
  A $C_2$-spectrum $E$ is homotopically even, according to our definition above, if and only if it is even in the sense of \cite[\S 3.1]{HillMeier}.
\end{rmk}

We prove the following theorem in \Cref{sec:Cmu-Purity}.

\begin{thm}
If a $p$-local $C_p$-ring spectrum $R$ is homotopically even, then it is also $\mu_p$-orientable.
\end{thm}

The key point here, as we explain in \Cref{sec:Cmu-Purity}, is that $\mathbb{CP}^{\infty}_{\mu_p}$ admits a \emph{slice cell decomposition} with even slice cells.  An even more fundamental fact, which turns out to be equivalent to the slice cell decomposition, is a splitting of the \emph{homology} of $\mathbb{CP}^{\infty}_{\mu_p}$:

\begin{dfn}
We say that a $C_p$-spectrum $X$ is \emph{homologically even} if there is a direct sum splitting
    \[X \otimes \HZpu \simeq \bigoplus_{k} A_k \otimes \HZpu,\]
where each $A_k$ is equivalent, for some $n \in \mathbb{Z}$, to one of
    \[(C_p)_+ \otimes S^{2n}, \, \, S^{2n\rho}, \, \, S^{2n\rho+1+\spoke}.\]
\end{dfn}

\begin{thm}
The space $\mathbb{CP}^{\infty}_{\mu_p}$ is homologically even.
\end{thm}

\begin{rmk}
The notion of homological evenness we propose in this paper restricts, when $p=2$, to the notion studied by Hill in \cite[Definition 3.2]{HillPurity}.  Notably, our definition differs from Hill's when $p>2$.

Returning again to the group $C_2$, work of Pitsch, Ricka, and Scherer relates a version of homological evenness to the study of \emph{conjugation spaces} \cite{ConjSpaces}.  An interesting example of a conjugation space, generalized in \cite{RWSI} and its in-progress sequel, is $\mathrm{BU}_{\mathbb{R}} = \Omega^{\infty} \Sigma^{\rho} \mathrm{BP}\langle 1 \rangle_{\mathbb{R}}$.  It would be very interesting to develop a $C_p$-equivariant version of conjugation space theory.  Since $\mathrm{tmf}(2)$ is a form of $\mathrm{BP} \langle 1 \rangle_{\mu_3}$ (cf. \Cref{BP1-qst}), we wonder whether there is an interesting slice cell decomposition of $\Omega^{\infty} \Sigma^{1+\Yright} \tmf(2)$.
\end{rmk}

\begin{rmk}
The slice cell structure on $\Sigma^{\infty}\mathbb{CP}^{\infty}_{\mu_p}$ has many interesting attaching maps.
The first non-trivial attaching map is a class $\alpha_1^{\mu_p}:S^{2\rho-1} \to S^{1+\Yright}$, with fixed points the multiplication by $p$ map on $S^1$.
This class was previously studied by the third author \cite[\S 3.2]{DWThesis} and, independently, Mike Hill.
The $C_2$-equivariant $\alpha_1^{\mu_2}$ is the familiar map $\eta:S^{\sigma} \to S^0$.
\end{rmk}

\subsection{A view to the future}

The most natural next question, after those tackled in this paper, is the following:
\begin{qst} \label{E-qst}
Let $n \ge 1$, and fix a formal group $\Gamma$ of height $n(p-1)$ over a perfect field $k$ of characteristic $p$.  When is the associated Lubin--Tate theory $E_{k,\Gamma}$ $\mu_p$-orientable?
\end{qst}
We have not fully answered this question even for $n=1$, since we focus attention on the Honda formal group.

It seems likely that further progress on \Cref{E-qst}, at least for $n \ge 2$, must wait for work in progress of Hill--Hopkins--Ravenel, who have a program by which to understand the $C_p$-action on Lubin--Tate theories. 
As the authors understand that work in progress, it is to be expected that the height $n(p-1)$ Morava $E$-theory has homotopy generated by $n$ copies of the reduced regular representation, $v_1^{\mu_p},v_2^{\mu_p},\cdots,v_{n}^{\mu_p}$.
 One expects to be able to construct $\mu_p$ Morava $K$-theories, generated by a single $v_i^{\mu_p}$, and we expect at least these Morava $K$-theories to be homotopically even in the sense of this paper.

\begin{qst}
Can one construct homotopically even $\mu_p$ Morava $K$-theories?
\end{qst}

In light of the orientation theory of \Cref{sec:Orientation}, it seems useful to know if $\mu_p$ Morava $K$-theories admit \emph{norms}.  Indeed, at $p=2$ the Real Morava $K$-theories all admit the structure of $\mathbb{E}_\sigma$-algebras.  Since the first $\mu_3$ Morava $K$-theory should be $\mathrm{TMF}(2)/3$, or perhaps $L_{K(2)} \mathrm{TMF}(2)/3$, it seems pertinent to answer the following question first:

\begin{qst}
At the prime $p=3$, what structure is carried by the $C_3$-equivariant spectrum $L_{K(2)} \mathrm{TMF}(2)/3$?  Is there an analog of the $\mathbb{E}_\sigma$ structure carried by $\mathrm{KU}_{\mathbb{R}}/2$?
\end{qst}

In another direction, one might ask about other finite subgroups of Morava stabilizer groups:
\begin{qst}
Is there an analog of the notion of $\mu_p$-orientation related to the $Q_8$-actions on Lubin--Tate theories at the prime $2$?
\end{qst}
One may also go beyond finite groups and ask for notions capturing other parts of the Morava stabilizer group, such as the central $\mathbb{Z}_p^{\times}$ that acts on $\mathbb{CP}^{\infty} \simeq \mathrm{B}^2\mathbb{Z}_p$ after $p$-completion.

To make full use of all these ideas, one would like not only an analog of $\mathbb{CP}^{\infty}_{\mathbb{R}}$, but also an analog of at least one of $\mathrm{MU}_{\mathbb{R}}$ or $\mathrm{BP}_{\mathbb{R}}$.
Attempts to construct such analogs have consumed the authors for many years; we consider it one of the most intriguing problems in stable homotopy theory today.

\begin{qst} (Hill--Hopkins--Ravenel \cite{HHRodd})
Does there exist a natural $C_p$-ring spectrum, $\mathrm{BP}_{\mu_p}$, with 
\begin{itemize}
\item Underlying, non-equivariant spectrum the smash product of $(p-1)$ copies of $\mathrm{BP}$.
\item Geometric fixed points $\Phi^{C_p} \mathrm{BP}_{\mu_p} \simeq \mathrm{H}\mathbb{F}_p$.
\end{itemize}
At $p=2$, it should be the case that $\mathrm{BP}_{\mu_2}=\mathrm{BP}_{\mathbb{R}}$.
\end{qst}
To the above we may add:
\begin{qst}
Does such a natural $\mathrm{BP}_{\mu_p}$ orient all $\mu_p$-orientable $C_p$-ring spectra, or at least all those that admit norms in the sense of \Cref{sec:Orientation}?
\end{qst}

Most of our attempts to build $\mathrm{BP}_{\mu_p}$ have proceeded via obstruction theory, while $\mathrm{MU}_{\mathbb{R}}$ is naturally produced via geometry.
It would be extremely interesting to see a geometric definition of an object $\mathrm{MU}_{\mu_{p}}$.
Alternatively, it would be very clarifying if one could prove that a reasonable $\mathrm{BP}_{\mu_p}$ does not exist.
As some evidence in that direction, the authors doubt any variant of $\mathrm{BP}_{\mu_p}$ can be homotopically even.

Even if $\mathrm{BP}_{\mu_p}$ cannot be built, or cannot be built easily, it would be excellent to know whether it is possible to build $C_p$-ring spectra $\mathrm{BP} \langle 1 \rangle_{\mu_p}$.

\begin{qst} \label{BP1-qst}
Does there exist, for each prime $p$, a $C_p$-ring $\mathrm{BP} \langle 1 \rangle_{\mu_p}$ satisfying the following properties:
\begin{itemize}
\item $\mathrm{BP} \langle 1 \rangle_{\mu_2}$ is the $2$-localization of $\mathrm{ku}_{\mathbb{R}}$, and $\mathrm{BP} \langle 1 \rangle_{\mu_3}$ is the $3$-localization of $\mathrm{tmf}(2)$.
\item The homotopy groups are given by 
\[\pi^{e}_* \mathrm{BP} \langle 1 \rangle_{\mu_p} \cong \mathbb{Z}_{(p)}[\lambda_1,\lambda_2,\cdots,\lambda_{p-1}],\] with $|\lambda_i|=2p-2$.
The $C_p$ action on these generators should make $\pi^{e}_{2p-2} \mathrm{BP} \langle 1 \rangle_{\mu_p}$ into a copy of the reduced regular representation.
\item There is a $C_p$-ring map $\mathrm{BP} \langle 1 \rangle_{\mu_p} \to E_{p-1}$.
\item $\mathrm{BP} \langle 1 \rangle_{\mu_p}$ is homotopically even, and in particular $\mu_p$-orientable.
\item The underlying spectrum $\left(\mathrm{BP} \langle 1 \rangle_{\mu_p}\right)^{e}$ additively splits into a wedge of suspensions of $\mathrm{BP} \langle p-1 \rangle$.
\item We have $\Phi^{C_p} \mathrm{BP} \langle 1 \rangle_{\mu_p} \simeq \F_p[y]$ for a generator $y$ of degree $2p$.
\end{itemize}
\end{qst}

It is plausible that $\mathrm{BP} \langle 1 \rangle_{\mu_p}$ should come in many forms, in the sense of Morava's forms of $K$-theory \cite{FormsK}.  A natural $\mathbb{E}_\infty$ form might be obtained by studying compactifications of the Gorbounov--Hopkins--Mahowald stack \cite{GorMah,HillThesis} of curves of the form
\[y^{p-1} = x(x-1)(x-a_1) \cdots (x-a_{p-2}).\]
Studying the uncompactified stack, it is possible to construct a $C_p$-equivariant $\mathbb{E}_\infty$ ring $E(1)_{\mu_p}$ which is a $\mu_p$ analog of uncompleted Johnson-Wilson theory.
The details of this construction will appear in forthcoming work of the second author.

\begin{rmk}
The $C_p$-action on $\mathbb{CP}^{\infty}_{\mu_p}$ is naturally the restriction of an action by $\Sigma_p$.
In fact, most objects in this paper admit actions of $\Sigma_p$, or at least of $C_{p-1} \ltimes C_p$, but these are consistently ignored.
The reader is encouraged to view this as an indication that the theory remains in flux, and welcomes further refinement.
\end{rmk}

\begin{rmk}
Since work of Quillen \cite{Quillen}, the notion of a complex orientation has been intimately tied to the notion of a formal group law.  There are hints throughout this paper, particularly in \Cref{sec:Orientation} and \Cref{sec:vot}, that the norm and diagonal maps on $\mathbb{CP}^{\infty}_{\mu_p}$ lead to equivariant refinements of the $p$-series of a formal group.  It may be interesting to develop the purely algebraic theory underlying these constructions, particularly if algebraically defined $v_i^{\mu_p}$ turn out to be of relevance to higher height Morava $E$-theories.
\end{rmk}

%
%

\addtocontents{toc}{\protect\setcounter{tocdepth}{0}}

\subsection{Notation and Conventions} \label{sec:Introduction:Notation} 
\begin{itemize}
  \item If $X$ is a $C_p$-space, we use $X^{e}$ to denote the underlying non-equivariant space, and we use $X^{C_p}$ to denote the fixed point space.  If $X$ is a $C_p$-spectrum, we will use either $\Phi^{e} X$ or $X^{e}$ to denote the underlying spectrum, and we use $\Phi^{C_p} X$ to denote the geometric fixed points.  

  \item We fix a prime number $p$, and throughout the paper all spectra and all (nilpotent) spaces are implicitly $p$-localized. In the $C_p$-equivariant setting, this means that we implicitly $p$-localize both underlying and fixed point spaces and spectra.
	
	\item If $X$ is a $C_p$-space or spectrum, we use $\pi^{e}_* X$ to denote the homotopy groups of $X^{e}$, considered as a graded abelian group \emph{with $C_p$-action}.  If $V$ is a $C_p$-representation, we use $\pi^{C_p}_{V} X$ to denote the set of homotopy classes of equivariant maps from $S^V$ to $X$.

  \item We let $S^{\spoke}$ denote the cofiber of the $C_p$-equivariant map $(C_p)_+ \to S^0$, and we also use $S^{\spoke}$ to refer to the suspension $C_p$-spectrum of this $C_p$-space. We let $S^{-\spoke}$ denote the Spanier-Whitehead dual of the $C_p$-spectrum $S^{\spoke}$.
    Given a $C_p$-representation $V$ and a $C_p$-spectrum $X$, we will use $\pi^{C_p} _{V+\spoke} X$ and $\pi^{C_p} _{V-\spoke}$ to denote the set of homotopy classes of equivariant maps from $S^{V+\spoke} \coloneqq S^V \otimes S^{\spoke}$ and $S^{V-\spoke} \coloneqq S^{V} \otimes S^{-\spoke}$ to $X$.

\item We let $\Cmu$ denote the fiber of the $C_p$-equivariant multiplication map $(\mathbb{CP}^{\infty})^{\times p} \to \mathbb{CP}^{\infty}$.

\item If $R$ is a classical commutative ring, we use $\bar{\rho}_{R}$ to denote the $R[C_p]$-module given by the augmentation ideal $\ker(R[C_p] \to R)$.  This is a rank $p-1$ $R$-module with generators permuted by the reduced regular representation of $C_p$.  We similarly use $\mathbbm{1}_{R}$ to denote the $R[C_p]$-module that is isomorphic to $R$ with trivial action.  We sometimes use $\rho_{R}$ to denote $R[C_p]$ itself, and write $\text{free}$ to denote a sum of copies of $\rho_{R}$.
\end{itemize}

\subsection{Acknowledgments} The authors thank Mike Hill, Mike Hopkins, and Doug Ravenel for inspiring numerous ideas in this document, as well as their consistent encouragement and interest in the work. We would especially like to thank Mike Hill for suggesting that our earlier definition of $\vot$ might be more conceptually viewed as a norm. We additionally thank Robert Burklund, Hood Chatham and Danny Shi for several useful conversations, and the anonymous referee for suggesting several improvements.

The first author was supported by NSF grant DMS-$1803273$, and thanks Yuzhi Chen and Wenyun Liu for their hospitality during the writing of this paper.  The second author was supported by an NSF GRFP fellowship under Grant No. 1745302. The third author was supported by NSF grant DMS-$1902669$.

\addtocontents{toc}{\protect\setcounter{tocdepth}{2}}

%% file: OrientationsFinal.tex
Non-equivariantly, one may study complex orientations of any unital spectrum $R$.  However, if $R$ is further equipped with the a homotopy commutative multiplication, then the theory takes on extra significance: in this case, a complex orientation of $R$ provides an isomorphism $R^*(\mathbb{CP}^{\infty}) \cong R_*[\![x]\!]$.

In this section, we work out the analogous theory for $\mu_p$-orientations.  In particular, we find that the theory of $\mu_p$-orientations takes on special significance for $C_p$ homotopy ring spectra $R$ that are equipped with a \emph{norm} $N_e^{C_p} R \to R$ refining the underlying multiplication.  Recall the following definition from the introduction:

\begin{dfn} A \emph{$\mu_p$-orientation} of a unital $C_p$-spectrum
$R$ is a map
	\[
	\Sigma^{\infty}\mathbb{CP}^{\infty}_{\mu_p} \longrightarrow
	\Sigma^{1+\Yright}R
	\]
such that the composite
	\[
	S^{1+\Yright} \longrightarrow
	\Sigma^{\infty}\mathbb{CP}^{\infty}_{\mu_p} \longrightarrow
	\Sigma^{1+\Yright}R
	\]
is equivalent to $\Sigma^{1+\Yright}$ of the unit.
\end{dfn}

For any $C_p$ representation sphere $S^V$, it is traditional to denote by $S^0[S^V]$ the free $\mathbb{E}_1$-ring spectrum
\[S^0[S^V] = S^0 \oplus S^V \oplus S^{2V} \oplus S^{3V} \oplus \cdots \]
Below, we extend this construction to take input not only representation spheres $S^V$, but spoke spheres as well.

\begin{dfn}
  For integers $n$, let
  \[S^0[S^{2n\rho-1-\spoke}] \coloneqq \NCp (S^0[S^{2np-2}]) \otimes_{S^0 [S^{2n\rho-2}]} S^0,\]
  where we consider $\NCp S^0 [S^{2np-2}]$ as a $S^0 [S^{2n\rho-2}]$-bimodule via the $\mathrm{E}_1$-map induced by the composite
%
  \[S^{2n\rho-2} \to (C_p)_+ \otimes S^{2np-2} \to \NCp S^0[S^{2np-2}].\]
  In this composite, the first map is adjoint to the identity on $S^{2np-2}$
  and the second map is the canonical inclusion.
	Note that $S^0[S^{2n\rho-1-\spoke}]$ is a unital left module over $\NCp S^0[S^{2np-2}]$.
	
  Furthermore, given a $C_p$-equivariant spectrum $R$, we set
  \[R[S^{2n\rho-1-\spoke}] \coloneq R \otimes S^0[S^{2n\rho-1-\spoke}].\]
\end{dfn}

\begin{cnstr} \label{cnstr-orientation}
Suppose that $R$ is a homotopy ring in
$C_p$-spectra, further equipped with a genuine norm map
	\[
	\NCp R \to R
	\]
which is unital and restricts on underlying spectra to the composite
  \[ (\Phi^e R)^{\otimes p} \xrightarrow{\mathrm{id} \otimes \gamma \otimes \dots \otimes \gamma^{p-1}} (\Phi^e R)^{\otimes p} \xrightarrow{m} \Phi^e R, \]
where $\gamma \in C_p$ is the generator and $m$ is the $p$-fold multiplication map.

If $R$ is $\mu_p$-oriented by a map
	\[
	S^{-1-\Yright} \to R^{\mathbb{CP}^{\infty}_{\mu_p+}}
	\]
then we may produce a map
	\[
	R[S^{-1-\Yright}] \to R^{\mathbb{C}P^{\infty}_{\mu_p+}}
	\]
as follows. First, the composite
	\[
	S^{-2} \stackrel{e}{\to} \Phi^{e}(C_{p+}\wedge S^{-2})
	\to \Phi^{e}(S^{-1-\Yright}) \to \Phi^{e}(R^{\mathbb{CP}^{\infty}_{\mu_p+}}),
	\]
  where the map $e$ is the inclusion of the factor of $S^{-2}$ corresponding to the identity in $C_p$, extends to a map
	\[
	S^0[S^{-2}] \to \Phi^{e}(R^{\mathbb{CP}^{\infty}_{\mu_p+}})
	\]
since the target is a homotopy ring. Norming up,
and combining the norm on $R$ with the diagonal map
$\mathbb{CP}^{\infty}_{\mu_p} \to 
\mathrm{Map}(C_p, \mathbb{CP}^{\infty}_{\mu_p})$, we get
a map
	\[
	\mathrm{N}_e^{C_p}(S^0[S^{-2}]) \to
	\mathrm{N}_e^{C_p}(R^{\mathbb{CP}^{\infty}_{\mu_p+}})
	\to R^{\mathbb{CP}^{\infty}_{\mu_p+}}.
	\]
Finally, the extension of $C_{p+} \wedge S^{-2} \to
R^{\mathbb{CP}^{\infty}_{\mu_p+}}$ over $S^{-1-\Yright}$
provides a nullhomotopy of the composite
  \[S^{-2} \to (C_p)_+ \otimes S^{-2} \to R^{\mathbb{CP}^{\infty}_{\mu_p+}},\]
  producing a map
	\[
	S^0[S^{-1-\Yright}] \to R^{\mathbb{CP}^{\infty}_{\mu_p+}}.
	\]
We finish by extending scalars to $R$, using the assumption that $R$ is a homotopy ring.
\end{cnstr}

\begin{cnstr} If $R$ is $\mu_p$-oriented then so too is the Postnikov
truncation
$R_{\le n}$. The construction above is natural, and so we may
form a map
	\[
	R[\![S^{-1-\Yright}]\!] :=
	\varprojlim R_{\le n}[S^{-1-\Yright}] \to
	\varprojlim (R_{\le n})^{\mathbb{CP}^{\infty}_{\mu_p+}}
	\simeq R^{\mathbb{CP}^{\infty}_{\mu_p+}}.
	\]
\end{cnstr}

\begin{thm} \label{thm:orient-homology}
  Suppose $R$ is a $\mu_p$-oriented homotopy $C_p$ ring, further equipped with a unital homotopy $N^{C_p}_e R$-module structure such that the unit
	\[N^{C_p}_e R \to R\]
	respects the underlying multiplication in the sense of \Cref{cnstr-orientation}.
Then, with notation as above, the map
	\[
	R[\![S^{-1-\Yright}]\!] \longrightarrow R^{\mathbb{CP}^{\infty}_{\mu_p +}}
	\]
is an equivalence.
\end{thm}

\begin{proof} By construction, it suffices to prove that the map
	\[
	R_{\le n}[S^{-1-\Yright}] \longrightarrow 
	(R_{\le n})^{\mathbb{CP}^{\infty}_{\mu_p +}}
	\]
is an equivalence for each $n\ge 0$. This is clear on underlying
spectra. On geometric fixed points we can factor this map as
	\[
    (\Phi^{C_p} R_{\le n})[S^{-1}] 
	\to
  (\Phi^{C_p} R_{\le n})^{\mathrm{B}C_{p+}}
	\to
  \Phi^{C_p} \left((R_{\le n})^{\mathbb{CP}^{\infty}_{\mu_p +}}\right),
	\]
being careful to interpret the source as a module (this is not a map
of rings).  Specifically, the above composite is one of unital $\Phi^{C_p} N^{C_p}_eS^0[S^{-2}] \simeq S^0[S^{-2}]$-modules and, separately, one of $\Phi^{C_p} R_{\le n}$-modules.

The second map is an equivalence by \Cref{lem:gfp-coh-comm} below, so
we need only prove the first map is an equivalence.
Since $p=0$ in $\Phi^{C_p} R$, the Atiyah-Hirzebruch spectral sequence
computing $\pi_*(\Phi^{C_p} R_{\le n})^{BC_{p+}}$ has $\mathrm{E}_2$-page given by
\[
  \pi_*(\Phi^{C_p} R_{\le n})  \otimes_{\mathbb{F}_p} \Lambda_{\F_p} (x) 
	\otimes_{\mathbb{F}_p}
	\mathbb{F}_p[y]
\]
The class $x$ is realized by applying geometric fixed point to the $\mu_p$-orientation.
The powers of $y$ are obtained from the unit of the unital $S^0[S^{-2}]$-module structure.
Using the $\Phi^{C_p} R_{\le n}$-module structure, this implies that the spectral sequence degenerates
and moreover that the first map is an equivalence.
\end{proof}

\begin{lem} \label{lem:gfp-coh-comm}
  If $R$ is bounded above,
and $X$ is a $C_p$-space of finite type, then the map
	\[
  (\Phi^{C_p} R)^{X^{C_p}_+}\to
  \Phi^{C_p} (R^{X_+})
	\]
is an equivalence.
\end{lem}
\begin{proof} Write $X = \colim X_n$ where the
$X_n$ are skeleta for a $C_p$-CW-structure on $X$
with each $X_n$ finite. Then the fiber of
	\[
    \Phi^{C_p} (R^{X_+}) \to \Phi^{C_p} (R^{X_{n+}})
	\]
becomes increasingly coconnective, and hence the map
	\[
    \Phi^{C_p} (R^{X_+}) \to
  \varprojlim \Phi^{C_p} (R^{X_{n+}})
	\]
is an equivalence. We are thus reduced to the case $X=X_n$
  finite, where the result follows since $\Phi^{C_p} (-)$ is exact.
\end{proof}

\begin{rmk}
In practice, the conditions of \Cref{thm:orient-homology} are often easy to check.  For example, any $C_p$-commutative ring $R$ in the homotopy category of $C_p$-spectra will be equipped with a unital homotopy $N_e^{C_p} R$-module structure respecting the multiplication in the sense of \Cref{cnstr-orientation} \cite{HillHopkins}.  We choose to write \Cref{thm:orient-homology} in generality because, even non-equivariantly, one occasionally studies orientations of rings that are not homotopy commutative (like Morava $K$-theory at the prime $2$). 
\end{rmk}

Since $\HZu$ is $\mu_p$-oriented by \Cref{cpmup-as-em-space} and truncated, we have the following corollary of \Cref{thm:orient-homology}:

\begin{cor} \label{cor:uZ-CPmup}
  There is a natural equivalence
  \[\HZu[S^{-1-\spoke}] \simeq \HZu^{\mathbb{CP}^{\infty} _{\mu_p+}}.\]
\end{cor}

%% file: Purity.tex
%
%
%
In this section, we will introduce a notion of \textit{evenness} in $C_p$-equivariant homotopy theory. This is a generalization of the notion of evenness in non-equivariant homotopy theory.
Evenness comes in two forms: \textit{homological} evenness and \textit{homotopical} evenness. Homological evenness is a $C_p$-equivariant version of the condition that a spectrum have homology concentrated in even degrees, and homotopical evenness corresponds to the condition that a spectrum have homotopy concentrated in even degrees.
%

The main results in this section are \Cref{prop:cell-structure}, which shows that, under certain conditions, a bounded below homologically even spectrum admits a cell decomposition into \textit{even slice spheres} (defined below), and \Cref{prop:even-unob}, which shows that there are no obstructions to mapping in a bounded below homologically even spectrum to a homotopically even spectrum.
%
%
%
%

\subsection{Homological Evenness}

We begin our discussion of evenness with the definition of an even slice sphere.

\begin{dfn} \label{dfn:even-slice-sphere}
    We say that a $C_p$-equivariant spectrum is an \textit{even slice sphere} if it is equivalent to one of the following for some $n \in \Z$:
    \[(C_p)_+ \otimes S^{2n}, \, \, S^{2n\rho}, \, \, S^{2n\rho+1+\spoke}.\]
    A \textit{dual even slice sphere} is the dual of an even slice sphere.
    The \textit{dimension} of a (dual) even slice sphere is the dimension of its underlying spectrum.
\end{dfn}

\begin{rmk}
  The phrase \emph{slice sphere} is taken from \cite[Definition 2.3]{DWSlices}, where a $G$-equivariant slice sphere is defined to be a compact $G$-equivariant spectrum, each of whose geometric fixed point spectra is a finite direct sum of spheres of a given dimension.

    It is easy to check that the (dual) even slice spheres of \Cref{dfn:even-slice-sphere} are slice spheres in this sense.
\end{rmk}

\begin{rmk}
    In the case $p=2$, the even slice spheres are precisely those of the form
    \[(C_2)_+ \otimes S^{2n}\text{ or } S^{n\rho}\]
    for some $n \in \Z$.
\end{rmk}

\begin{dfn}
    We say that a $C_p$-equivariant spectrum $X$ is \textit{homologically even} if there is an equivalence of $\HZpu$-modules
    \[X \otimes \HZpu \simeq \bigoplus_{n} S_n \otimes \HZpu,\]
    where $S_n$ is a direct sum of even slice spheres of dimension $2n$.
\end{dfn}

\begin{rmk}
  When $p=2$, this recovers the notion of homological purity given in \cite[Definition 3.2]{HillPurity}. However, when $p$ is odd, our definition of homological evenness differs from Hill's definition of homological purity.
  The most important difference is that we allow the spoke spheres $S^{2n\rho+1+\spoke}$ to appear in our definition. This is necessary for $\Cmu$ to be homologically even.
\end{rmk}

As in the non-equivariant case, homological evenness for a bounded below spectrum is equivalent to the existence of an even cell structure. To prove this, we need to recall the following definition from \cite{DWSlices, HY}:

\begin{dfn}
  We say that a $C_p$-equivariant spectrum $X$ is \emph{regular slice $n$-connective} if:
    \begin{enumerate}
      \item $X^e$ is $n$-connective, and
      \item $\gfp X$ is $\lceil \frac{n}{p} \rceil$-connective.
    \end{enumerate} 
    %
%
%
%
    Furthermore, we say that $X$ is \emph{bounded below} if it is regular slice $n$-connective for some integer $n$.
\end{dfn}

\begin{lem} \label{lem:slice-HF-check}
  Let $X$ be a bounded below $C_p$-spectrum with the property that $\Phi^{C_p} X$ is of finite type. Then $X$ is regular slice $n$-connective if and only if $X \otimes \HZpu$ is regular slice $n$-connective.
\end{lem}

\begin{proof}
  For the underlying spectrum, the follows from the fact that $\HZp$ detects connectivity of bounded below $p$-local spectra. For the geometric fixed points, we use the fact that $\gfp \HZpu = \HFp[y]$, $\abs{y} = 2$, detects connectivity of bounded below $p$-local spectra which are of finite type, since a finitely generated $\Z_{(p)}$-module is trivial if and only if it is trivial after tensoring with $\F_p$.
\end{proof}

\begin{lem} \label{lem:slice-sphere-conn}
    Let $W$ denote an even slice sphere of dimension $n$, and suppose that $X$ is regular slice $n$-connective. Then we have $[W, \Sigma X] = 0$.
\end{lem}

\begin{proof}
    If $W$ is of dimension $n$, then its underlying spectrum $W^e$ is a direct sum of $n$-spheres and $\gfp W$ is a $\lceil \frac{n}{p} \rceil$-sphere. It therefore follows that $W$ is a regular slice $n$-sphere in the sense of \cite[\S 2.1]{DWSlices}, so the conclusion follows from \cite[Proposition 2.22]{DWSlices}.
\end{proof}

\begin{prop}\label{prop:cell-structure}
  Suppose that $X$ is a bounded below, homologically even $C_p$-equivariant spectrum with the property that $\Phi^{C_p} X$ is of finite type, so that there exists a splitting
  \[X \otimes \HZpu \simeq \bigoplus_{k \geq n} S_k \otimes \HZpu,\]
  where $S_k$ is a direct sum of $2k$-dimensional even slice spheres. Then $X$ admits a filtration $\{X_k\}_{k \geq n}$ such that $X_k / X_{k-1} \simeq S_k$ for each $k \geq n$.
\end{prop}

\begin{proof}
%
%
    By assumption, we are given a splitting
    \[X \otimes \HZpu \simeq \bigoplus_{k \geq n} S_k \otimes \HZpu,\]
    where $S_k$ is a direct sum of $2k$-dimensional even slice spheres.
    By induction on $n$, it will suffice to show that the dashed lifting exists in the diagram
    \begin{center}
    \begin{tikzcd}
        & X \arrow[d] \\
        S_n \arrow[r] \arrow[ur, dashrightarrow] & \bigoplus_{k \geq n} S_k \otimes \HZpu \simeq X \otimes \HZpu,
    \end{tikzcd}
    \end{center}
    since the cofiber of any such lift is a bounded below homologically even $C_p$-spectrum with $\Phi^{C_p} X$ of finite type and whose $\HZpu$-homology is $\bigoplus_{k \geq n+1} S_k \otimes \HZpu$.

    Note that \Cref{lem:slice-HF-check} implies that $X$ is regular slice $2n$-connected. Let $F$ be the fiber of the Hurewicz map $S^0 \to \HZpu$. Then $F$ is easily seen to be regular slice $0$-connective,
    so that $F \otimes X$ is regular slice $2n$-connective.
    This implies that $[S_n, \Sigma F \otimes X] = 0$ by \Cref{lem:slice-sphere-conn}.
    The result now follows from the cofiber sequence
    \[X \to \HZpu \otimes X \to \Sigma F \otimes X. \qedhere\]
%
%
%
%
\end{proof}

\begin{rmk}
  It will follow from \Cref{exm:hom-pure} and \Cref{prop:pure-splitting} that the following converse of \Cref{prop:cell-structure} holds: if $X$ is bounded below and admits an even slice cell structure, then $X$ is homologically even.
\end{rmk}

\subsection{Homotopical Evenness}

We now introduce the homotopical version of evenness.

\begin{dfn} \label{dfn:hmtpy-pure}
    We say that a $C_p$-equivariant spectrum $E$ is \textit{homotopically even} if the following conditions hold for all $n \in \Z$:
    \begin{enumerate}
        \item $\pi_{2n-1} ^{e} E = 0.$
        \item $\pi_{2n\rho-1} ^{C_p} E = 0.$
        \item $\pi_{2n\rho-2-\spoke} ^{C_p} E = 0.$
    \end{enumerate}
\end{dfn}

\begin{rmk}
  All of the examples of homotopically even $C_p$-spectra that we will encouter will satisfy the following condition for all $n \in \Z$:
  \begin{enumerate}
      \setcounter{enumi}{3}
    \item $\pi_{2n\rho+\spoke} ^{C_p} E = 0.$
  \end{enumerate}
  We will say that a homotopically even $C_p$-spectrum \emph{satisfies condition (4)} if this holds.

  In fact, the examples which we study satisfy even stronger evenness properties. We have chosen the weakest possible set of properties for which our theorems hold.
\end{rmk}

\begin{rmk} \label{rmk:hmtp-pure-rephrase}
    If we assume condition (1), then we may rewrite conditions (3) and (4) as follows:
    \begin{enumerate}
        \item[($3'$)] the transfer maps $\pi^{e} _{2np-2} E \to \pi^{C_p} _{2n\rho-2} E$ are surjective for all $n \in \Z$.
        \item[($4'$)] the restriction maps $\pi^{C_p} _{2n\rho} E \to \pi^{e} _{2np} E$ are injective for all $n \in \Z$.
    \end{enumerate}
    This follows directly from the cofiber sequences defining $S^{-\spoke}$ and $S^\spoke$:
    \[S^{-\spoke} \to S^0 \xrightarrow{\mathrm{tr}} (C_p)_+ \otimes S^{0}\]
    \[(C_p)_+ \otimes S^{0} \xrightarrow{\mathrm{res}} S^0 \to S^{\spoke}.\]
\end{rmk}

\begin{rmk} \label{rmk:slice-comment}
Conditions $(1)$-$(4)$ have some implications for the slice tower of any homotopically even $E$, which can be read off from \cite{HY} (cf., \cite[3.5]{DWSlices}).  First, conditions $(1)$ and $(2)$ together imply that slices in degrees $2np-1$ are trivial.  Secondly, condition $(4)$ implies that the $(2np)$th slice is the zero-slice determined by the Mackey functor $\pi_{2n\rho}$.  However, when $p>2$ the implication of $(3)$ for slices is obscure, and many slices are unconstrained.
\end{rmk}

\begin{rmk}
    If $p=2$, \Cref{dfn:hmtpy-pure} reduces to the requirement that, for all $n \in \Z$:
    \begin{enumerate}
        \item $\pi_{2n-1} ^{e} E = 0.$
        \item $\pi_{n\rho-1} ^{C_2} E = 0.$
    \end{enumerate}
    A $C_2$-equivariant spectrum is therefore homotopically even if and only if it is \emph{even} in the sense of \cite[Definition 3.1]{HillMeier}.
    Moreover, condition (4) is redundant in the $C_2$-equivariant setting.
\end{rmk}

\begin{exm} \label{exm:hom-pure}
  The Eilenberg--Maclane spectra $\HFu$ and $\HZpu$ are examples of homotopically even $C_p$-spectra which satisfy condition (4). To verify this, we refer to the reader to the appendix of third author's thesis \cite[\S A]{DWThesis}, where one may find a computation of the spoke graded homotopy groups of $\HFu$ and $\HZpu$.

    At the prime $p=2$, there are many examples of homotopically even $C_2$-spectra in the literature, such as $\MU_\R, \BP_\R, \BPn_\R, E(n)_\R$, $K(n)_\R$ and $E_n$, where $E_n$ is equipped with the Goerss-Hopkins $C_2$-action \cite{HillMeier,RealOrEn}.

    The main result of \Cref{sec:Hmtpy-Purity} is that the $C_p$-spectra $E_{p-1}$ and the $C_3$-spectrum $\tmf(2)$ are homotopically even and satisfy condition (4).
\end{exm}

When trying to map a bounded below homologically even $C_p$-spectrum into a homotopically even $C_p$-spectrum, there are no obstructions:

\begin{prop} \label{prop:even-unob}
  Let $E$ be a homotopically even $C_p$-spectrum, and suppose that $X$ is a $C_p$-spectrum equipped with a bounded below filtration $\{X_k\}_{k \geq n}$ such that each $S_k \coloneqq X_k / X_{k-1}$ is a direct sum of $2k$-dimensional even slice spheres.

  Then, for any $k \geq n$, every $C_p$-equivariant map $X_k \to E$ extends to an equivariant map $X \to E$.
\end{prop}

\begin{proof}
  It suffices to prove by induction that any map $X_k \to E$ extends to a map $X_{k+1} \to E$.
  Using the cofiber sequence
  \[\Sigma^{-1} S_{k+1} \to X_k \to X_{k+1},\]
  we just need to know that any map from the desuspension of an even slice sphere into $E$ is nullhomotopic.
  This follows precisely from the definition of homotopical evenness.
\end{proof}

If $E$ further satisfies condition (4), we have the stronger result:

\begin{prop} \label{prop:pure-splitting}
  Let $E$ be a homotopically even $C_p$-ring spectrum which satisfies condition (4), and suppose that
    $X$ is a $C_p$-spectrum equipped with a bounded below filtration $\{X_k\}_{k \geq n}$ such that each $S_k \coloneqq X_k / X_{k-1}$ is a direct sum of $2k$-dimensional even slice spheres.

    Then there is a splitting of the induced filtration on $X \otimes E$ by $E$-modules:
    \[X \otimes E \simeq \bigoplus_{k \geq n} S_k \otimes E.\]
\end{prop}

\begin{proof}
    We need to show that the filtration $\{X_k\}_{k \geq n}$ splits upon smashing with $E$.
    Working by induction, we see that it suffices to show that all maps
    \[S_k \to \Sigma S_m \otimes E,\]
    where $k > m$, are automatically null.
    Enumerating through all of the possible even slice spheres that can appear in $S_k$ and $S_m$, and making use of the (non-canonical) equivalence
    \[S^{\spoke} \otimes S^{-\spoke} \simeq S^0 \oplus \bigoplus_{p-2} \left( (C_p)_+ \otimes S^0 \right),\]
    we find that this follows precisely from the hypothesis that $E$ is homotopically even and satisfies condition (4).
\end{proof}

%% file: CP_purity.tex
The main goal of this section is to prove the following theorem:

\begin{thm}\label{thm:CPmup-pure}
    The $C_p$-spectrum $\Sigma^\infty \Cmu$ is homologically even.
\end{thm}

Noting that $\Phi^{C_p} \Sigma^\infty \Cmu = \Sigma^\infty BC_p$ is of finite type, we may apply \Cref{prop:cell-structure} and so deduce the following corollary:

\begin{cor} \label{cor:CPmup-cells}
    There is a filtration $\{\Sigma^\infty \CP^{n} _{\mu_p}\}_{n \geq 0}$ of $\Sigma^\infty \Cmu$ with subquotients as follows
    \[\Sigma^\infty \Cmun / \Sigma^\infty \CP^{n-1} _{\mu_p} \simeq \begin{cases} S^{2m\rho} \oplus \bigoplus ((C_p)_+ \otimes S^{2n}), & \text{if } n = mp \\ S^{2m\rho+1+\spoke} \oplus \bigoplus ((C_p)_+ \otimes S^{2n}), & \text{if } n = mp+1\\ \bigoplus ((C_p)_+ \otimes S^{2n}), &\text{otherwise}.  \end{cases}\]
\end{cor}

\begin{wrn}
  We believe that there is a filtration $\{\Cmun\}_{n \geq 0}$ of the space $\Cmu$ that recovers $\{\Sigma^\infty \CP^{n} _{\mu_p}\}_{n \geq 0}$ upon applying $\Sigma^\infty$, but we do not prove this here. As such, our name $\Sigma^\infty \Cmun$ must be regarded as an abuse of notation: we do not prove that $\Sigma^{\infty} \Cmun$ is $\Sigma^\infty$ of a $C_p$-space $\Cmun$.  In light of the Dold-Thom theorem, it seems likely that the space $\Cmun$ could be defined as the $n$th symmetric power of $S^{1+\Yright}$.
\end{wrn}

\begin{rmk}
    The identification of the particular even slice spheres appearing in this decomposition is determined by the cohomology of $\Cmu$ as a $C_p$-representation, and in particular from the combination of \Cref{cor:uZ-CPmup}, \Cref{lem:hom-sym} and \Cref{prop:sym-rhobar}.
\end{rmk}

As an application, we obtain the following analog of the fact that any ring spectrum with homotopy groups concentrated in even degrees admits a complex orientation:

\begin{cor} \label{cor:pure-oriented}
    Let $E$ be a homotopically even $C_p$-ring spectrum. Then $E$ is $\mu_p$-orientable.
\end{cor}

\begin{proof}
    We wish to show that that the $(1+\spoke)$-suspension of the unit map factors as
    \[S^{1+\spoke} \to \Sigma^{\infty} \Cmu \to \Sigma^{1+\spoke} E.\]
  This is an immediate consequence of \Cref{cor:CPmup-cells} and \Cref{prop:even-unob}.
\end{proof}

We devote the remainder of the section to the proof of \Cref{thm:CPmup-pure}.
By \Cref{cor:uZ-CPmup}, there is an equivalence
\[\HZu[S^{-1-\spoke}] \simeq \HZu^{\mathbb{CP}^{\infty} _{\mu_p+}}.\]
This is of finite type, so to prove \Cref{thm:CPmup-pure} it will suffice to prove the following theorem and dualize:

\begin{thm} \label{thm:spoke-alg-splitting}
    As a $C_p$-equivariant spectrum, $S^0 [S^{2n\rho-1-\spoke}]$ is a direct sum of dual even slice spheres for all $n \in \Z$.
\end{thm}

To prove this, we will construct a map in from a wedge of dual even slice spheres which is an equivalence on underlying spectra and geometric fixed points.

\begin{cnst}
    The composition
    \[S^{2n\rho-2} \to (C_p)_+ \otimes S^{2np - 2} \to \NCp S^0 [S^{2np-2}] \to S^0 [S^{2n\rho-1-\spoke}]\]
    is canonically null, and hence induces a map
    \[\wt{x} : S^{2n\rho-1-\spoke} \to S^0 [S^{2n\rho-1-\spoke}].\]
    On the other hand, letting \[x : S^{2np-2} \to S^0[S^{2np-2}]\] denote the canonical inclusion, there is the norm map
    \[\Nm(x) : S^{(2np-2)\rho} \to \NCp S^0 [S^{2np-2}] \to S^0 [S^{2n\rho-1-\spoke}].\]
    Since $S^0 [S^{2n\rho-1-\spoke}]$ is a module over $\NCp S^0 [S^{2np-2}]$, this implies the existence of maps
    \[\Nm(x)^{k} \cdot \wt{x}^\eps : S^{k(2np-2)\rho+\eps(2n\rho-1-\spoke)} \to S^0 [S^{2n\rho-1-\spoke}]\]
    for $k \in \N$ and $\eps \in \{0,1\}$.
\end{cnst}

We first show that the sum of these maps induces an equivalence on geometric fixed points:

\begin{prop} \label{prop:geo-fixed-gen}
    Let
    \[\Psi : \bigoplus_{\stackrel{k\geq 0}{\eps \in \{0,1\}}} S^{k(2np-2)\rho+\eps(2n\rho-1-\spoke)} \to S^0 [S^{2n\rho-1-\spoke}]\]
    denote the direct sum of the maps $\Nm(x)^k \cdot \wt{x}^\eps$.
    Then $\gfp (\Psi)$ is an equivalence.
\end{prop}

\begin{proof}
We have an identification
\[\gfp S^0 [S^{2n\rho-1-\spoke}] \simeq S^0 [S^{2np-2}] \otimes_{S^0 [S^{2n-2}]} S^0 \simeq S^0 [S^{2np-2}] \otimes (S^0 \otimes_{S^0 [S^{2n-2}]} S^0).\]
Under this identification, the map
\[\gfp (\Nm (x)) : S^{2np-2} \to \gfp S^0 [S^{2n\rho-1-\spoke}]\]
corresponds to the inclusion of $S^{2np-2}$ into the left factor.

There are equivalences
  \[ S^0 \otimes_{S^0 [S^{2n-2}]} S^0 \simeq \Sigma^\infty _+ \Bar(*, \Omega \Sigma S^{2n-2}, *) \simeq \Sigma^\infty _+ S^{2n-1}\]
and hence an isomorphism
\[ \H_* ^e (S^0 \otimes_{S^0[S^{2n-2}]} S^0; \Z) \cong \Lambda_{\Z} (x_{2n-1}).\]
Furthermore, the map
\[\gfp (\wt{x}) : S^{2n-1} \to \gfp S^0 [S^{2n\rho-1-\spoke}]\]
sends the fundamental class of $S^{2n-1}$ to $x_{2n-1}$.

It follows that $\gfp (\Psi)$ induces an isomorphism on homology, so is an equivalence.
\end{proof}

Our next task is to extend $\Psi$ to a map that also induces an equivalence on underlying spectra. We will see that this can be accomplished by taking the direct sum with maps from induced even spheres, which are easy to produce. The main input is a computation of the homology of the underlying spectrum of $S^0 [S^{2n\rho-1-\spoke}]$ as a $C_p$-representation.

\begin{lem} \label{lem:hom-sym}
    There is a $C_p$-equivariant isomorphism
    \[\H_* ^e (S^0 [S^{2n\rho-1-\spoke}]; \Z) \cong \Sym^* _{\Z} (\rhobar),\]
    where $\rhobar$ lies in degree $(2np-2)$.
\end{lem}

\begin{proof}
    There are equivariant isomorphisms
    \[\H_* ^e (S^0 [S^{2n\rho-2}]; \Z) \cong \Sym^* _{\Z} (x)\]
    and
    \[\H_* ^e (\NCp S^0 [S^{2np-2}]; \Z) \cong \Sym^* _{\Z} (\rho),\]
    where $x$ and $\rho$ both lie in degree $(2np-2)$.
    Since $S^0 [S^{2n\rho-1-\spoke}]$ is a unital $\NCp S^0 [S^{2np-2}]$-module, we obtain a map
    \[\Sym^* _{\Z} (\rho) \to \H_* ^e (S^0 [S^{2n\rho-1-\spoke}]; \Z)\]
    of $\Sym^* _{\Z} (\rho)$-modules.
    Since $x$ goes to zero in $\H_* ^e (S^0 [S^{2n\rho-1-\spoke}]; \Z)$, it follows that this factors through a map
    \[ \Sym^* _{\Z} (\rhobar) \cong \Sym^* _{\Z} (\rho) \otimes_{\Sym^* _{\Z} (x)} \Z \to \H_* ^e (S^0 [S^{2n\rho-1-\spoke}]; \Z).\]
    Examining the K\"unneth spectral sequence, we see that this map must be an isomorphism.
%
%
\end{proof}


The following theorem in pure algebra determines the structure of the mod $p$ reduction $\Sym^* _{\F_p} (\rhobar)$ as a $C_p$-representation:

\begin{prop}[{\cite[Propositions III.3.4-III.3.6]{AlmFoss}}] \label{prop:sym-rhobar}
  Let $\rhobar$ denote the reduced regular representation of $C_p$ over $\F_p$, and let $e_1, \dots e_p \in \rhobar$ denote generators which are cyclically permuted by $C_p$ and satisfy $e_1 + \dots + e_p = 0$. We set $\Nm = e_1 \cdots e_p \in \Sym^p _{\F_p} (\rhobar)$.
    Then the symmetric powers of $\rhobar$ decompose as follows:
    \begin{align*}
        \Sym^k _{\F_p} (\rhobar) \cong
        \begin{cases}
            \triv\{\Nm^{\ell}\} \oplus \mathrm{free}& \text{ if } k = \ell \cdot p \\
            \rhobar\{\Nm^{\ell}e_1, \dots, \Nm^{\ell}e_p\}\oplus \mathrm{free}& \text{ if } k = \ell \cdot p + 1 \\
            \mathrm{free}& \text{ otherwise.} 
        \end{cases}
    \end{align*}
\end{prop}
%
%
%
%

\begin{proof}[Proof of \Cref{thm:spoke-alg-splitting}]
  Let $\Psi$ be as in \Cref{prop:geo-fixed-gen}. It follows from \Cref{lem:hom-sym} and \Cref{prop:sym-rhobar} that the mod $p$ homology of $\Phi^e (S^0 [S^{2n\rho-1-\spoke}])$ splits as $\im(\H_{*} ^e (\Psi)) \oplus \mathrm{free}$. Moreover, $\Psi$ is an equivalence on geometric fixed points by \Cref{prop:geo-fixed-gen}.

    It therefore suffices to show that, given any summand of $\H_{2k} ^e (S^0 [S^{2n\rho-1-\spoke}]; \F_p)$ isomorphic to $\rho$, there is a map $(C_p)_+ \otimes S^{2k} \to S^0 [S^{2n\rho-1-\spoke}]$ whose image is that summand. Taking the direct sum of $\Psi$ with an appropriate collection of such maps, we obtain an $\F_p$-homology equivalence. Since both sides have finitely-generated free $\Z$-homology, this must in fact be a $p$-local equivalence, as desired.

    To prove the remaining claim, it suffices to show that the mod $p$ Hurewicz map
    \[\pi^{e} _{*} (S^0 [S^{2n\rho-1-\spoke}]) \to \H_{*} ^e (S^0 [S^{2n\rho-1-\spoke}]; \F_p)\]
    is surjective in every degree. This follows from the following square
    \begin{center}
        \begin{tikzcd}
            \pi ^{e} _* (\NCp S^0 [S^{2np-2}]) \arrow[r, two heads] \arrow[d] & \H_* ^e (\NCp S^0 [S^{2np-2}]; \F_p) \arrow[d, two heads]\\
            \pi^{e} _{*} (S^0 [S^{2n\rho-1-\spoke}]) \arrow[r] & \H_{*} ^e (S^0 [S^{2n\rho-1-\spoke}]; \F_p),
        \end{tikzcd}
    \end{center}
    where the top horizontal arrow is a surjection because $\NCp S^0 [S^{2np-2}]$ is a non-equivariant direct sum of spheres, and the right vertical arrow is a surjection by the proof of \Cref{lem:hom-sym}.
%
%
%
\end{proof}

%% file: Hom_purity_examples.tex
In this section, we introduce our principal examples of homotopically even $C_p$-ring spectra. By \Cref{cor:pure-oriented}, they are also $\mu_p$-orientable.

Our first examples are the the Morava $E$-theories $E_{p-1}$ associated to the height $p-1$ Honda formal group. As we will recall in \Cref{sec:Ethy}, $E_{p-1}$ admits an essentially unique $C_p$-action by $\mathbb{E}_\infty$-automorphisms. We use this action to view $E_{p-1}$ as a Borel $C_p$-equivariant $\mathbb{E}_\infty$-ring.

Our second example is the connective $\mathbb{E}_\infty$-ring $\tmf(2)$ of topological modular forms with full level $2$ structure. The group $\GL_2 (\Z / 2\Z) \cong \Sigma_3$ acts on $\tmf(2)$ via modification of the level $2$ structure, and we view $\tmf(2)$ as a $C_3$-equivariant $\mathbb{E}_\infty$-ring via the inclusion $C_3 \subset \Sigma_3$. We will discuss this example in \Cref{sec:tmf2}.

The main result of this section is the homotopical evenness of the above $C_p$-ring spectra:

\begin{thm} \label{thm:E-thy-pure}
  The Borel $C_p$-equivariant height $p-1$ Morava $E$-theories $E_{p-1}$ associated to the Honda formal group over $\F_{p^{p-1}}$ are homotopically even and satisfy condition (4).
\end{thm}

\begin{thm} \label{thm:tmf2-pure}
  The $C_3$-ring spectrum $\tmf(2)$ of connective topological modular forms with full level $2$ structure is homotopically even and satisfies condition (4).
\end{thm}

Applying \Cref{cor:pure-oriented}, we obtain the following corollary:

\begin{cor}
    The $C_p$-ring spectra $E_{p-1}$ and $\tmf(2)$ are $\mu_p$-orientable.
\end{cor}
%

\subsection{Height $p-1$ Morava $E$-theory} \label{sec:Ethy}
Given a pair $(k, \G)$, where $k$ is a perfect field of characterstic $p > 0$ and $\G$ is a formal group $\G$ over $k$ of finite height $h$, we may functorially associate an $\mathbb{E}_\infty$-ring $E(k,\G)$, the Lubin-Tate spectrum or Morava $E$-theory spectrum of $(k,\G)$ \cite{GHObst,ECII}.
There is a non-canonical isomorphism
\[\pi_{*} E(k,\G) \cong \mathbb{W}(k) \llbracket u_1, \dots, u_{h-1} \rrbracket [u^{\pm 1}],\]
where $\abs{u_i} = 0$ and $\abs{u} = -2$.


Given a prime $p$ and finite height $h$, a formal group particularly well-studied in homotopy theory is the Honda formal group.
The Honda formal group $\G_h ^{\mathrm{Honda}}$ is defined over $\F_p$, so the Frobenius isogney may be viewed as a endomorphism
\[F : \G_h ^{\mathrm{Honda}} \to \G_h ^{\mathrm{Honda}}.\]
The Honda formal group is uniquely determined by the condition that $F^h = p$ in $\End(\G_h ^{\mathrm{Honda}})$.

The endomorphism ring of the base change of $\G_h ^{\Honda}$ to $\F_{p^h}$ is the maximal order $\OO_h$ in the division algebra $\DD_h$ of Hasse invariant $1/h$ and center $\Q_p$. By the functoriality of the Lubin-Tate theory construction, the automorphism group $\SS_h = \OO_h ^\times$ of $\G_h ^{\Honda}$ over $\F_{p^h}$ acts on $E(\F_{p^h}, \G_h ^\Honda)$. To keep our notation from becoming too burdensome, we set
\[E_{p-1} \coloneqq E(\F_{p^{p-1}}, \G_{p-1} ^\Honda).\]

There is a subgroup $C_p \subset \SS_{p-1}$, which is unique up to conjugation. Indeed, such subgroups correspond to embeddings $\QQ_p (\zeta_p) \subset \DD_{p-1}$. Since $\QQ_p (\zeta_p)$ is of degree $p-1$ over $\QQ_p$, it follows from a general fact about division algebras over local fields that such a subfield exists and is unique up to conjugation (cf. \cite[Application on pg. 138]{SerreLCF}).
Using any such $C_p$, we may view $E_{p-1}$ as a Borel $C_p$-equivariant $\mathbb{E}_\infty$-ring spectrum.
%

Homotopical evenness of $E_{p-1}$ will follow from the computation of the homotopy fixed point spectral sequence for $E_{p-1} ^{h C_p}$, which was first carried out by Hopkins and Miller and has been written down in \cite{NaveEO} and again reviewed in \cite{PicEO}. We recall this computation below. The homotopy fixed point spectral sequence takes the form
\[\H^s (C_p, \pi_{t} E_{p-1}) \Rightarrow \pi_{t-s} E_{p-1} ^{h C_p},\]
so the first step is to compute the action of $C_p$ on $\pi_* E_{p-1}$.

This action may be determined as follows. Abusing notation, let $v_1 \in \pi_{2p-2} E_{p-1}$ denote a lift of the canonically defined element $v_1 \in \pi_{2p-2} E_{p-1}/p$. The element $v_1$ is fixed modulo $p$ by the $\SS_{p-1}$ and in particular the $C_p$-action on $E_{p-1}$, so if we fix a generator $\gamma \in C_p$ we find that the element $v_1 - \gamma v_1$ is divisible by $p$.
Set $v = \frac{v_1 - \gamma v_1}{p}$. Then the two key properties of $v$ are that:
\begin{enumerate}
    \item $v + \gamma v + \dots + \gamma^{p-1} v = 0$.
    \item $v$ is a unit in $\pi_* E_{p-1}$. As a consequence, $\Nm(v) = v \cdot \gamma v \cdots \gamma^{p-1} v$ is a unit in $\pi_* E_{p-1}$ which is fixed by the $C_p$-action \cite[pg. 498]{NaveEO}.
\end{enumerate}
The existence of an element $v$ satisfying the above two conditions completely determines the action of $C_p$ on $\pi_* E_{p-1}$, as follows. First, let $\wt{w} \in \pi_{-2} E_{p-1}$ denote any unit, and set $w = v \cdot \Nm (\wt{w}) \in \pi_{-2} E_{p-1}$. Then $w$ continues to satisfy (1) and (2) above and determines a map of $C_p$-representations
\[\rhobar_{\mathbb{W}(\F_{p^{p-1}})} \to \pi_{-2} E_{p-1}.\]
This determines a $C_p$-equivariant map
\[\Sym^* _{\mathbb{W}(\F_{p^{p-1}})} (\rhobar) [\Nm(w)^{-1}] \to \pi_* E_{p-1},\]
which identifies $\pi_{*} E_{p-1}$ with the graded completion of $\Sym^* _{\mathbb{W}(\F_{p^{p-1}})} (\rhobar) [\Nm(w)^{-1}]$ at the graded ideal generated by the kernel of the essentially unique nonzero map of $\mathbb{W}(\F_{p^{p-1}})[C_p]$-modules $\rhobar_{\mathbb{W}(\F_{p^{p-1}})} \to \triv_{\F_{p^{p-1}}}$.

\begin{rmk} \label{rmk:Ep-1-forward}
  In \Cref{sec:Image}, we will see that the element $v$ is intimately related to the $\mu_p$-orientability of $E_{p-1}$.
  For later use, we note that it follows from the above analysis that the map $\rhobar_{\F_{p^{p-1}}} \to \pi_{2p-2} E_{p-1} /(p, \mathfrak{m}^2)$ induced by $v$ is an isomorphism.
\end{rmk}

\begin{rmk}
  As pointed out by the referee, the element $v \in \pi_{2p-2} E_{p-1}$ may also be described in terms of $\BP$-theory.
  The class $t_1 \in \BP_{2p-2} \BP$ determines a function $t_1 : \mathbb{S}_{p-1} \to E_{2p-2}$, and it follows from the formula $\eta_R (v_1) = v_1 + p t_1$ in $\BP_* \BP$ that $t_1 (\gamma^{-1}) = \frac{v_1-\gamma v_1}{p} = v$.
  From this perspective, the crucial fact that $v$ is a unit in $\pi_* E_{p-1}$ follows from the calculations in \cite[pgs. 438-439]{Rav78}.
\end{rmk}

Using the above determination of the $C_p$-action on $\pi_* E_{p-1}$, as well as \Cref{prop:sym-rhobar}, one may obtain with some work the following description of $\H^s (C_p, \pi_t E_{p-1})$:

\begin{prop} [{Hopkins--Miller, cf. \cite[Proposition 2.6]{PicEO}}]
  There is an exact sequence
  \begin{align} \label{eq:SES}
    \pi_* E_{p-1} \xrightarrow{tr} \H^* (C_p, \pi_* E_{p-1}) \to \F_{p^{p-1}} [\alpha,\beta,\delta^{\pm 1}]/(\alpha^2 ) \to 0,
  \end{align}
  where $\abs{\alpha}  = (1,2p-2), \abs{\beta} = (2,2p^2-2p),$ and $\abs{\delta} = (0,2p)$.
\end{prop}

Finally, we must recall the differentials in the homotopy fixed point spectral sequence. We let $\doteq$ denote equality up to multiplication by an element of $\mathbb{W}(\F_{p^{p-1}})^\times$. Then, as explained in \cite[\S 2.4]{PicEO}, the spectral sequence is determined multiplicatively by the following differentials:
\[d_{2(p-1)+1} (\delta) \doteq \alpha \beta^{p-1} \delta^{1-(p-1)^2} \,\,\, \text{ and } \,\,\, d_{2(p-1)^2+1} (\delta^{(p-1)^3} \alpha) \doteq \beta^{(p-1)^2+1},\]
along with the fact that all differentials vanish on the image of the transfer map.

In particular, on the $\mathrm{E}_\infty$-page of the homotopy fixed point spectral sequence there are no elements in positive filtration in total degrees $0$, $-1$ or $-2$. Indeed, there are no elements at all in the $(-1)$-stem.

%
%

We now have enough information to establish the homotopical evenness of $E_{p-1}$.

\begin{proof}[Proof of \Cref{thm:E-thy-pure}]
  Let $u \in \pi^{e} _{2} E_{p-1}$ denote the periodicity element. Then $\Nm(u) \in \pi^{C_p} _{2\rho} E_{p-1}$ is also invertible, so the $RO(C_p)$-graded equivariant homotopy of $E_{p-1}$ is $2\rho$-periodic.

    Therefore, using \Cref{rmk:hmtp-pure-rephrase}, we see that it suffices to show that:
    \begin{enumerate}
        \item $\pi_{-1} ^e E_{p-1} = 0$.
        \item $\pi_{-1} E_{p-1} ^{hC_p} = 0$.
        \item The transfer map $\pi_{-2} ^e E_{p-1} \to \pi_{-2} E_{p-1} ^{hC_p}$ is a surjection.
        \item The restriction map $\pi _{0} E_{p-1} ^{hC_p} \to \pi_{0} ^e E_{p-1}$ is an injection.
    \end{enumerate}

    Condition (1) is immediate from the fact that $E_{p-1}$ is even periodic.
    Condition (2) is a direct consequence of the above computation of the homotopy fixed point spectral sequence.
    Condition (3) follows from the following two facts:
    \begin{itemize}
      \item The short exact sequence (\ref{eq:SES}) implies that $\H^0 (C_p, \pi_{-2} E_{p-1})$ is spanned by the image of the transfer.
        \item On the $\mathrm{E}_{\infty}$-page of the homotopy fixed point spectral sequence, there are no positive filtration elements in stem $-2$.
    \end{itemize}
    Condition (4) follows from the fact that on the $\mathrm{E}_\infty$-page of the homotopy fixed point spectral sequence, there are no positive filtration elements in the zero stem.
%
\end{proof}

\subsection{The spectrum $\tmf(2)$ as a form of $\BP\langle1\rangle_{\mu_3}$} \label{sec:tmf2}
%
%
%
%
%
%
Recall from \cite{Vesnatmf2} or \cite{tmflevel} the spectrum $\tmf(2)$ of connective topological modular forms with full level $2$ structure.\footnote{The spectrum $\tmf(2)$ is obtained from the spectrum $\Tmf(2)$ discussed in the references by taking the $\Sigma_3$-equivariant connective cover.} In this section we will consider $\tmf(2)$ as implictly $3$-localized. It is a genuine $\Sigma_3$-equivariant $\mathbb{E}_\infty$-ring spectrum with $\Sigma_3$-fixed points $\tmf(2)^{\Sigma_3} = \tmf$, the ($3$-localized) spectrum of connective topological modular forms. We view $\tmf(2)$ as a $C_3$-spectrum via restriction along an inclusion $C_3 \subset \Sigma_3$.

This spectrum has been well-studied by Stojanoska \cite{Vesnatmf2}. In particular, Stojanoska computes $\pi_{*} ^e \tmf(2) = \Z_{(3)}[\lambda_1, \lambda_2]$, where $\abs{\lambda_i} = 4$ and a generator $\gamma$ of $C_3$ acts by $\lambda_1 \mapsto \lambda_2 -\lambda_1$ and $\lambda_2 \mapsto - \lambda_1$. It follows that $\lambda_1$ and $\lambda_2$ span a copy of $\rhobar$, so that $\pi_* \tmf(2) \cong \Sym^* _{\Zlt} (\rhobar)$.
The corresponding family of elliptic curves is cut out by the explicit equation
\[y^2 = x(x-\lambda_1)(x-\lambda_2).\]
For later use, we note down some facts about the associated formal group law.

\begin{prop} \label{tmf(2)formulae}
  The $3$-series of the formal group law associated to $\tmf(2)$ is given by the following formula:
  \begin{align*}
    [3] (x) &= 3x + 8(\lambda_1 + \lambda_2)x^3 + 24(\lambda_1^2 - 2\lambda_1 \lambda_2 + \lambda_2 ^2)x^5 + 72(\lambda_1 ^3 - \lambda_1 ^2 \lambda_2 - \lambda_1 \lambda_2 ^2 + \lambda_2 ^3)x^7 \\
            &+ 8(27\lambda_1 ^4 - 76 \lambda_1 ^3 \lambda_2 + 98 \lambda_1 ^2 \lambda_2 ^2 - 76\lambda_1 \lambda_2 ^3 + 27 \lambda_2 ^4)x^9 + O(x^{10})
  \end{align*}
  It follows that we have the following formulas for $v_1$ and $v_2$:
  \[v_1 \equiv - \lambda_1 - \lambda_2 \mod 3\]
  and
  \[v_2 \equiv \lambda_1 ^4 \equiv \lambda_2 ^4 \mod (3,v_1).\]
\end{prop}

\begin{proof}
  This is an elementary computation using the method of \cite[\S IV.1]{Silverman}.
\end{proof}

\begin{rmk} \label{tmf(2)calc}
  Let $v = -\lambda_1 - \lambda_2$, so that $v \equiv v_1 \mod p$. Then we have
  \[\gamma v - v = \left( (\lambda_1 - \lambda_2) + \lambda_1 \right) + \lambda_1 + \lambda_2 = 3 \lambda_1,\]
  so that
  \[\frac{\gamma v - v}{3} = \lambda_1.\]
  Note that this element generates $\pi_* \tmf(2)$ as a $\Zlt$-algebra with $C_3$-action.
  In \Cref{sec:Image}, we will relate this element to the $\mu_3$-orientation of $\tmf(2)$.
\end{rmk}

In his thesis, the third author has
computed the slices of $\tmf(2)$ (cf. \cite[\S 4]{HHR}):

\begin{prop}[{\cite[Corollary 3.2.1.10]{DWThesis}}]\label{prop:tmf2-slices}
  Given a $C_p$-equivariant spectrum $X$, let $P^n _n X$ denote the $n$th slice of $X$.
  The slices of $\tmf(2)$ are of the form:
  \[\bigoplus_{n} P^n _n \tmf(2) \simeq \HZu_{(3)}[S^{2\rho-1-\spoke}].\]
\end{prop}

We now turn to the proof of \Cref{thm:tmf2-pure}. Given the computation of the slices of $\tmf(2)$ in \Cref{prop:tmf2-slices}, this will follow from \Cref{thm:spoke-alg-splitting} and the following proposition:
%

\begin{prop} \label{prop:htpy-pure-slice}
  Let $X$ be a $C_p$-spectrum whose slices are of the form $P_n ^n X \simeq S_n \otimes \HZu_{(p)}$, where $S_n$ is a direct sum of dual even slice $n$-spheres. Then $X$ is homotopically even and satisfies condition (4).
\end{prop}

Using the slice spectral sequence, the proof of \Cref{prop:htpy-pure-slice} reduces to the following lemma:

\begin{lem} \label{lem:dual-slice-pure}
  Let $S$ denote a dual even slice sphere. Then $S \otimes \HZu_{(p)}$ is homotopically even and satisfies condition (4).
\end{lem}

\begin{proof}
  If $S \simeq S^{2n} \otimes (C_p)_+$, then this follows from the fact that $\pi_{2n-1} \HZ_{(p)} = 0$ for all $n \in \Z$.

    If $S \simeq S^{2n\rho}$, then this follows from the fact that $\HZu_{(p)}$ is homotopically even, since the definition of homotopically even is invariant under $2\rho$-suspension.

    If $S \simeq S^{2n\rho - 1 -\spoke}$, then condition (1) of \Cref{dfn:hmtpy-pure} is clearly satisfied, and conditions (2)-(4) follow from the following statements for all $n \in \Z$, which may be read off from \cite[\S A.2]{DWThesis}:
    \begin{itemize}
      \setcounter{enumi}{2}
      \item $\pi^{C_p} _{2n\rho+\spoke} \HZu_{(p)} = 0$,
      \item $\pi^{C_p} _{2n\rho-1} \HZu_{(p)} = 0$,
      \item $\pi^{C_p} _{2n\rho+1+\lambda} \HZu_{(p)} = 0,$
    \end{itemize}
    where in the proofs of (3) and (4) we have implicitly used the existence of equivalences
    \[S^{\spoke} \otimes S^{-\spoke} \simeq S^{0} \oplus \bigoplus_{p-2} (C_p)_+ \otimes S^0\]
    and
    \[S^{\spoke} \otimes S^{\spoke} \simeq S^{\lambda} \oplus \bigoplus_{p-2} (C_p)_+ \otimes S^2.\qedhere\]
\end{proof}

%
%
%
%

%% file: Vot.tex
In this section, given a $\mu_p$-oriented $C_p$-ring spectrum $R$, we will define a class 
\[\vot \in \pi^{C_p} _{2\rho} (\Sigma^{1+\spoke} R) \cong \pi^{C_p} _{2\rho - 1 -\spoke} R.\] 
When $p=2$, our construction agrees with the class $v_1 ^{\R} \in \pi^{C_2} _\rho R$ in the homotopy of a Real oriented $C_2$-ring spectrum.
Just as $v_1$ is well-defined modulo $p$, we will see that $\vot$ is well-defined modulo the transfer.
We will also give a formula for the image of $\vot$ in the the underlying homotopy of $R$ in terms of the classical element $v_1$ and the $C_p$-action.

To define $\vot$, we first construct a class $\vot \in \pi^{C_p} _{2\rho} \Sigma^\infty \Cmu$, and then we take its image along the $\mu_p$-orientation $\Sigma^\infty \Cmu \to \Sigma^{1+\spoke} R$. To begin, we recall an analogous construction of the classical element $v_1$.

\subsection{The non-equivariant $v_1$ as a $p$th power}

We recall some classical, non-equivariant theory that we will generalize to the equivariant setting in the next section.

\begin{ntn}
We let $\beta:S^2 \simeq \Sigma^{\infty}\mathbb{CP}^1 \to \Sigma^{\infty}\mathbb{CP}^{\infty}$ denote a generator of the stable homotopy group $\pi_2(\Sigma^{\infty} \mathbb{CP}^{\infty})$.
\end{ntn}

Since $\mathbb{CP}^{\infty} \simeq \Omega^{\infty} \Sigma^2 \mathbb{Z}$ is an infinite loop space, its suspension spectrum $\Sigma^{\infty} \mathbb{CP}^{\infty}$ is a non-unital ring spectrum.  This allows us to make sense of the following definition.

\begin{dfn} \label{defclassicalv1}
We define the class $v_1 \in \pi_{2p} \Sigma^{\infty} \mathbb{CP}^{\infty}$ to be $\beta^p$, the $p$th power of the degree $2$ generator.
\end{dfn}

There are at least two justifications for naming this class $v_1$, which might more commonly be defined as the coefficient of $x^p$ in the $p$-series of a complex-oriented ring.  The relationship is expressed in the following proposition:

\begin{prop} \label{prop:p-series}
Let $R$ denote a (non-equivariant) homotopy ring spectrum, equipped with a complex orientation
$$\Sigma^{-2} \Sigma^{\infty} \mathbb{CP}^{\infty} \to R,$$
which can be viewed as a class $x \in R^2(\mathbb{CP}^{\infty})$.
Then the composite
$$S^{2p-2} \stackrel{v_1}{\to} \Sigma^{-2} \Sigma^{\infty} \mathbb{CP}^{\infty} \to R$$
records, up to addition of a multiple of $p$, the coefficient of $x^p$ in the $p$-series $[p]_F(x)$.
\end{prop}

\begin{proof}
Consider the $p$-fold multiplication map of infinite loop spaces
$$(\mathbb{CP}^{\infty})^{\times p} \stackrel{m}{\to} \mathbb{CP}^{\infty}$$
Applying $R^*$ to the above, we obtain a map
$$R_* \llbracket x \rrbracket \to R_* \llbracket x_1,x_2,\cdots, x_p \rrbracket.$$
By the definition of the formal group law $\--+_F\--$ associated to the complex orientation, the class $x \in R^2(\mathbb{CP}^{\infty})$ is sent to the formal sum 
$$f(x_1,x_2,\cdots,x_p)=x_1+_F x_2 +_F \cdots +_F x_p.$$  The commutativity of the formal group law ensures that this power series is invariant under cyclic permutation of the $x_i$.

The composite 
$$S^{2p-2} \to \Sigma^{-2} (\Sigma^{\infty} \mathbb{CP}^{\infty})^{\otimes p} \to \Sigma^{-2} \Sigma^{\infty} \mathbb{CP}^{\infty} \to R$$
that we must compute can be read off as the coefficient of the product $x_1x_2\cdots x_p$ in the power series $f(x_1,x_2,\cdots,x_p)$.
We may of course consider other degree $p$ monomials in the $x_i$, such as $x_1^p$.  The coefficient in $f(x_1,x_2,\cdots,x_p)$ of any such degree $p$ monomial will be an element of $\pi_{2p-2}R$.  Summing these coefficients over all the possible degree $p$ monomials, we obtain the the coefficient of $x^p$ in the single variable power series $[p]_F(x)=f(x,x,\cdots,x)$.

Our claim is that this sum differs from the coefficient of $x_1x_2\cdots x_p$ by a multiple of $p$.  The reason is that $x_1x_2\cdots x_p$ is the unique monomial invariant under the cyclic permutation of the $x_i$.  For example, the coefficients of $x_1^p,x_2^p,\cdots,$ and $x_n^p$ will all be equal, so their sum is a multiple of $p$.
\end{proof}

\begin{rmk}
The integral homology $H_*(\mathbb{CP}^{\infty};\mathbb{Z}_{(p)})$ is a divided power ring on the Hurewicz image of $\beta$.  In particular, the Hurewicz image of $v_1=\beta^p$ is a multiple of $p$ times a generator of $H_{2p}(\mathbb{CP}^{\infty};\mathbb{Z}_{(p)})$.

Consider the ring spectrum $MU$ together with its canonical complex orientation
  $$\Sigma^{-2} \Sigma^{\infty} \mathbb{CP}^{\infty} \to MU.$$
The integral homology $H_*(MU;\mathbb{Z})$ is the symmetric algebra on the image, under this map, of $\widetilde{H}_*(\mathbb{CP}^{\infty};\mathbb{Z})$.
  In particular, the Hurewicz image of $v_1$ in $H_{2p}(\Sigma^{-2} \Sigma^{\infty} \mathbb{CP}^{\infty};\mathbb{Z}_{(p)})$ is sent to $p$ times an indecomposable generator of $H_{2p-2}(MU;\mathbb{Z}_{(p)})$.  By \cite{Milnor}, this provides another justification for the name $v_1$.
\end{rmk}

\begin{rmk}
One might ask whether higher $v_i$, with $i>1$, can be defined in $\pi_*(\Sigma^{\infty}\mathbb{CP}^{\infty})$.  A classical argument with topological $K$-theory \cite{Mosher} shows that the Hurewicz image of $\pi_*(\Sigma^{\infty} \mathbb{CP}^{\infty})$ inside of $H_*(\Sigma^{\infty} \mathbb{CP}^{\infty};\mathbb{Z}_{(p)})$ is generated as a $\mathbb{Z}_{(p)}$-module by powers of $\beta$.  For $i$ larger than $1$, $\beta^{p^i}$ is not simply $p$ times a generator of $H_{2p^i}(\mathbb{CP}^{\infty};\mathbb{Z}_{(p)})$, so it is impossible to lift the corresponding indecomposable generators of $\pi_*(MU)$ to $\pi_*(\Sigma^{-2} \Sigma^{\infty} \mathbb{CP}^{\infty})$.  However, it may be possible to lift multiples of such generators.
\end{rmk}

Finally, we record the following proposition for later use:

\begin{prop} \label{divbyp}
Let $A$ denote a (non-equivariant) homotopy ring spectrum, equipped with a map
$$f:\Sigma^{\infty} \mathbb{CP}^{\infty} \to \Sigma^2 A$$
that induces the \emph{zero} homomorphism on $\pi_2$ (in particular, $f$ is not a complex orientation).  Then the image of $v_1$ in $\pi_{2p-2} A$ is a multiple of $p$.
\end{prop}

\begin{proof}
Let $C\alpha_1$ denote the cofiber of $\alpha_1:S^{2p-3} \to S^0$.

We recall first that, $p$-locally, the spectrum 
$$\Sigma^{\infty} \mathbb{CP}^{p}$$ 
admits a splitting as $\Sigma^2 C\alpha_1 \oplus \bigoplus_{k=2}^{p-2} S^{2k}$.  Indeed, since $\alpha_1$ is the lowest positive degree element in the $p$-local stable stems, most of the attaching maps in the standard cell structure for $\mathbb{CP}^{p}$ are automatically $p$-locally trivial.  The only possibly non-trivial attaching map is between the $(2p)$th cell and the bottom cell, and this attaching map is detected by the $P^1$ action on $H^*(\mathbb{CP}^{\infty};\mathbb{F}_p)$.

By cellular approximation, $v_1:S^{2p} \to \Sigma^{\infty} \mathbb{CP}^{\infty}$ must factor through $\Sigma^{\infty} \mathbb{CP}^{p}$, and again the lack of elements in the $p$-local stable stems ensures a further factorization of $v_1$ through $\Sigma^2 C\alpha_1$.  Thus, to determine the image of $v_1$ in $\pi_{2p}(\Sigma^2 A)$, it suffices to consider the composite
$$\tilde{f}:\Sigma^2 C\alpha_1 \to \Sigma^{\infty} \mathbb{CP}^{p} \to \Sigma^{\infty} \mathbb{CP}^{\infty} \to \Sigma^2 A.$$

There is by definition a cofiber sequence $S^2 \to \Sigma^2 C\alpha_1 \to S^{2p}$.  By the assumption that $f$ is trivial on $\pi_2$, $\tilde{f}$ must factor as a composite
$$\Sigma^2 C\alpha_1 \to S^{2p} \to \Sigma^2 A.$$
We now finish by noting that the composite $v_1:S^{2p} \to \Sigma^2 C\alpha_1 \to S^{2p}$ must be a multiple of $p$, because otherwise $C\alpha_1$ would split as $S^{2p} \oplus S^2$.
\end{proof}

\begin{rmk}
The argument used in the proof of \Cref{divbyp} suggests yet another interpretation of \Cref{prop:p-series}, as pointed out by the referee.  \Cref{prop:p-series} is true because $\pi_{2p-2}C\alpha_1$ is generated by $v_1$ in the Adams-Novikov spectral sequence.
\end{rmk}

\subsection{The equivariant $\vot$ as a norm}

As we defined the non-equivariant $v_1 \in \pi_{2p} \Sigma^{\infty} \mathbb{CP}^{\infty}$ to be the $p$th power of a degree $2$ class, we similarly define an equivariant $\vot \in \pi^{C_p}_{2\rho} \Sigma^{\infty} \Cmu$ to be the \emph{norm} of a degree $2$ class.  We thank Mike Hill for suggesting this conceptual way of constructing $\vot$.  To see that $\Sigma^\infty \Cmu$ is equipped with norms, we will make use of the following proposition:

\begin{prop}\label{prop:cpmup-as-em-space}
There is an equivalence of $C_p$-equivariant spaces
$$\Omega^{\infty} \Sigma^{1+\Yright} \underline{\mathbb{Z}} \simeq \Cmu,$$
where $\underline{\mathbb{Z}}$ denotes the $C_p$-equivariant Eilenberg--Maclane spectrum associated to the constant Mackey functor.  
\end{prop}

\begin{proof}
This is \Cref{cpmup-as-em-space}.
\end{proof}

\begin{cnstr} \label{Cmunorm}
The above proposition equips the space $\mathbb{CP}^{\infty}_{\mu_p}$ with a natural \emph{norm}, meaning a map
$$N^{C_p}_{e}((\Cmu)^e) \to \Cmu.$$
Indeed, any $C_p$-equivariant infinite loop space $\Omega^{\infty} Y$, like $\Omega^{\infty} S^{1+\Yright} \underline{\mathbb{Z}}$, is equipped with a norm
$$N^{C_p}_{e} (\Omega^{\infty} Y)^e \to \Omega^{\infty} Y.$$
This norm is $\Omega^{\infty}$ applied to the $C_p$-spectrum map
$$(C_p)_+ \otimes Y \to Y$$
that is induced from the identity on $Y^e$.
\end{cnstr}

\begin{cnv} \label{coordinates}
For the remainder of this section we fix a (non-canonical) equivalence
$$(\Cmu)^{e} \simeq (\mathbb{CP}^{\infty})^{\times p-1}.$$
The natural map of $C_p$-spaces
  $$S^{1+\Yright}=\mathbb{CP}^1_{\mu_p} \to \mathbb{CP}^{\infty}_{\mu_p}$$
then induces an (again, non-canonical) equivalence
  $$(S^{1+\Yright})^e \simeq \bigvee_{p-1} S^2,$$
giving $p-1$ classes 
$$\beta_1,\beta_2,\dots,\beta_{p-1} \in \pi^{e}_{2}(\mathbb{CP}^{\infty}_{\mu_p}).$$
Choosing our non-canonical equivalence appropriately, we may suppose that the $C_p$-action on $\pi^{e}_2(\mathbb{CP}^{\infty}_{\mu_p};\mathbb{Z}_{(p)})$ is given by the rules
\begin{enumerate}
\item $\gamma (\beta_i) = \beta_{i+1}$, if $1 \le i \le p-2$
\item $\gamma (\beta_{p-1}) = -\beta_1-\beta_2-\cdots-\beta_{p-1}.$
\end{enumerate}
\end{cnv}

\begin{dfn}
We let
$$\vot:S^{2\rho} \to \Sigma^{\infty} \Cmu.$$
denote the \emph{norm} of $\beta_1$.  Explicitly, norming the non-equivariant $\beta_1$ map yields a map
  $$S^{2\rho} \simeq N_e^{C_p} S^2 \to N_e^{C_p}(\Phi^e (\Sigma^{\infty} \Cmu)),$$
and we may compose this with the norm map of \Cref{Cmunorm} to make the class
$$\vot \in \pi^{C_p}_{2\rho}(\Sigma^{\infty} \Cmu).$$
\end{dfn}

\begin{rmk}
Of course, the choice of the class $\beta_1$ above is not canonical.  We view this as a mild indeterminancy in the definition of $\vot$, related to the fact that the classical $v_1$ should only be well-defined modulo $p$.  As we will see later, many formulas we write for $\vot$ will similarly be well-defined only modulo transfers.
\end{rmk}



\subsection{A formula for $\vot$ in terms of $v_1$}

Our next aim will be to give an explicit formula for the \emph{image} of $\vot$ in the underlying homotopy of a $\mu_p$-oriented cohomology theory.  Our formula is stated as \Cref{thm:votFormula}.  To begin its derivation, our first order of business is to give a different formula for $\vot$ modulo transfers:

\begin{prop} \label{votalt}
In $\pi^{e}_{2p}(\Sigma^{\infty}\Cmu)$, the class $p\vot$ and the class $\Tr(\beta_1^p)$ differ by $p$ times a transferred class.  In particular, $\Tr(\beta_1^p)$ is divisible by $p$, and the class $\frac{\Tr(\beta_1^{p})}{p}$ is the restriction of a class in $\pi^{C_p}_{2\rho} \Sigma^{\infty} \Cmu$.
\end{prop}

\begin{proof}
  Identifying $\pi_{2} ^e (\Sigma^\infty \Cmu)$ with $\rhobar_{\Z_{(p)}}$ and using the nonunital $\mathbb{E}_\infty$-ring structure on $\Sigma^\infty \Cmu$, we obtain a map
  \[\Sym^{p} _{\Z_{(p)}} (\rhobar_{\Z_{(p)}}) \to \pi_{2p} ^e (\Sigma^\infty \Cmu)\]
  under which the norm class $\Nm$ maps to the image of $\vot$.
  The conclusion of the proposition then follows from \Cref{lem:tr-nm} below.
%
%
\end{proof}

\begin{lem} \label{lem:tr-nm}
  Let $\rhobar_{\Z_{(p)}}$ denote the reduced regular representation of $C_p$ over $\Z_{(p)}$, and let $e_1, \dots, e_p \in \rhobar_{\Z_{(p)}}$ denote generators which are cyclically permuted by $C_p$ and satisfy $e_1 + \dots + e_p = 0$. We set $\Nm = e_1 \cdots e_p \in \Sym^p _{\Z_{(p)}} (\rhobar_{\Z_{(p)}})$.

  Then $\Tr(e_1 ^p)$ is divisible by $p$, and $\Nm$ and $\frac{\Tr(e_1 ^p)}{p}$ differ by a transferred class in $\Sym^p _{\Z_{(p)}} (\rhobar_{\Z_{(p)}})$.
\end{lem}

\begin{proof}
  To see that $\Tr (e_1 ^p)$ is divisible by $p$, we expand it out in terms of the basis $e_1, \dots, e_{p-1}$ of $\rhobar_{\Z_{(p)}}$:
  \[\Tr(e_1 ^p) = e_1 ^p + \dots + e_{p-1} ^p + (-e_1 - e_2 - \dots - e_{p-1})^p .\]
  It is clear from linearity of the Frobenius modulo $p$ that $\Tr(e_1 ^p)$ is divisible by $p$.
  Our next goal is to show that
  $\Nm - \frac{\Tr(e_1 ^p)}{p}$
  is a transferred class. It is clearly fixed by the $C_p$-action, so we wish to show that its image in
  \[\frac{\left( \Sym^p _{\Z_{(p)}} (\rhobar_{\Z_{(p)}}) \right)^{C_p}}{\Tr \left( \Sym^p _{\Z_{(p)}} (\rhobar_{\Z_{(p)}}) \right)} \]
  is zero.
  Since $p$ times any fixed point of $C_p$ is the transfer of an element, there is an isomorphism
  \[\frac{\left( \Sym^p _{\Z_{(p)}} (\rhobar_{\Z_{(p)}}) \right)^{C_p}}{\Tr \left( \Sym^p _{\Z_{(p)}} (\rhobar_{\Z_{(p)}}) \right)} \cong \frac{\left( \Sym^p _{\F_{p}} (\rhobar_{\F_p}) \right)^{C_p}}{\Tr \left( \Sym^p _{\F_{p}} (\rhobar_{\F_p}) \right)}.\]
  By \Cref{prop:sym-rhobar}, there is an isomorphism of $C_p$-representations
  \[ \Sym^p _{\F_{p}} (\rhobar_{\F_p}) \cong \triv_{\F_p} \{\Nm\} \oplus \mathrm{free}, \]
  so that any choice of $C_p$-equivariant map
  $\Sym^p _{\F_{p}} (\rhobar_{\F_p}) \to \triv_{\F_p}$
  which is nonzero on $\Nm$ restricts to an isomorphism
  \[ \frac{\left( \Sym^p _{\F_{p}} (\rhobar_{\F_p}) \right)^{C_p}}{\Tr \left( \Sym^p _{\F_{p}} (\rhobar_{\F_p}) \right)} \cong \triv_{\F_p}.  \]

  A choice of such a map may be made as follows. First, let $f : \rhobar_{\F_p} \to \triv_{\F_p}$ denote the equivariant map sending each $e_i$ to $1$. This induces a map $\Sym^p _{\F_p} (f) : \Sym^p _{\F_{p}} (\rhobar_{\F_p}) \to \Sym^p _{\F_{p}} (\triv_{\F_p}) \cong \triv_{\F_p}$ which sends $\Nm$ to $1$. We now need to show that the image of
  $\frac{\Tr(e_1 ^p)}{p}$
  under $\Sym^p _{\F_p} (f)$ is also equal to $1$. Writing
  \[\frac{\Tr(e_1 ^p)}{p} = \frac{e_1 ^p + \dots + e_{p-1} ^p + (-e_1 - e_2 - \dots - e_{p-1})^p}{p} ,\]
  we find that its image of $\Sym^p _{\F_p} (f)$ is equal to
  \[\frac{p-1 - (p-1)^p}{p} = \frac{p-1 - (-1 + O(p^2))}{p} \equiv 1 \mod p,\]
  as desired.
\end{proof}

\Cref{votalt} can be read as the statement that $\frac{\Tr(\beta_1^{p})}{p}$ is a formula for $\vot \in \pi_{2\rho}^{C_p} \Sigma^{\infty} \Cmu$, if one is only interested in $\vot$ modulo transfers.  We often find this formula for $\vot$ to be more useful in computational contexts.

\begin{cnv}
For the remainder of this section, we fix a $C_p$-ring $R$ together with a $\mu_p$-orientation
$$\Sigma^{\infty} \Cmu \to \Sigma^{1+\Yright} R.$$
\end{cnv}

\begin{dfn} \label{undorientation}
The $\mu_p$-orientation of $R$ gives rise to a map
  $$(\Sigma^{\infty} \mathbb{CP}^{\infty}_{\mu_p})^e \to (\Sigma^{1+\Yright} R)^e ,$$
which under our fixed identification of $(\mathbb{CP}^{\infty}_{\mu_p})^e$ is given by a map
  $$\Sigma^{\infty} (\mathbb{CP}^{\infty})^{\times p-1} \to \bigoplus_{p-1} \Sigma^2 R.$$
By mapping in the first of the $(p-1)$ copies of $\mathbb{CP}^{\infty}$, and then projecting to the first of the $(p-1)$ copies of $R$, we obtain \emph{the underlying complex orientation of $R$}.
\end{dfn}

\begin{wrn}
While it is convenient to give formulas in terms of the underlying complex orientation of \Cref{undorientation}, we stress once again that this is non-canonical, depending on \Cref{coordinates}.  There is no \emph{canonical} classical complex orientation associated to a $\mu_p$-oriented $C_p$-ring.  
\end{wrn}

\begin{ntn} \label{extractclassicalv1}
Using \Cref{defclassicalv1}, the underlying complex orientation of $R$ gives rise to a class $v_1=\beta_1^p\in \pi^{e}_{2p-2} R$.
\end{ntn}

\begin{ntn}
  Recall our fixed non-canonical identification $(S^{1+\Yright})^e \simeq \bigoplus_{p-1} S^2$.  Let $y_i \in \pi_{2} ^e S^{1+\spoke}$ correspond to the $i$th copy of $S^2$, so that we have
  \begin{enumerate}
    \item $\gamma (y_i) = y_{i+1}$ if $1 \leq i \leq p-2$, and
    \item $\gamma (y_{p-1}) = -y_1 - \dots - y_{p-1}$.
  \end{enumerate}
  Then a generic class
  $$r \in \pi_{2p}^e(\Sigma^{1+\Yright} R) \cong \pi_{2} ^e S^{1+\spoke} \otimes \pi_{2p-2} ^e R$$
may be written as
$$r = y_1 \otimes r_1 + y_2 \otimes r_2 + \cdots + y_{p-1} \otimes r_{p-1},$$
where $r_i \in \pi^{e}_{2p-2} R$.
\end{ntn}

The key relationship between the equivariant $\vot$ and non-equivariant $v_1$ is expressed in the following lemma:

\begin{lem} \label{lem:v1-spoke}
The class $v_1=\beta_1^p \in \pi^{e}_{2p} \Sigma^{\infty} \Cmu$ maps to $y_1 \otimes v_1$ plus a multiple of $p$ in $\pi^{e}_{2p}(\Sigma^{1+\Yright} R).$
\end{lem}

\begin{proof}
The class $\beta_1^p$ maps to $y_1 \otimes r_1 + y_2 \otimes r_2 + \cdots + y_{p-1} \otimes r_{p-1}$ for some collection of elements $r_1,r_2,...,r_{p-1} \in \pi^{e}_{2p-2} R$.
  
By \Cref{defclassicalv1}, $r_1=v_1$, so it suffices to show that each of $r_2,...,r_{p-1}$ is divisible by $p$.  These statements in turn each follow by application of \Cref{divbyp}.
\end{proof}

At last, we are ready to state the main result of this section:

\begin{thm} \label{thm:votFormula}
Suppose that the underlying homotopy groups $\pi_*^{e} R$ are torsion-free.  Then the  class $\vot \in \pi_{2p}^{e}(\Sigma^{1+\Yright} R)$ is given, modulo transfers, by the class
$$y_1 \otimes \frac{v_1-\gamma^{p-1} v_1}{p} + y_2 \otimes \frac{\gamma v_1-v_1}{p} + \cdots + y_{p-1} \otimes \frac{\gamma^{p-2} v_1 - \gamma^{p-3} v_1}{p}.$$
\end{thm}

\begin{proof}
  By \Cref{votalt}, it is equivalent to show the above formula determines $\Tr(\beta_1 ^p)/p \in \pi_{2p} ^e (\Sigma^{1+\spoke} R)$ modulo transfers. But this may be computed directly from \Cref{lem:v1-spoke}.
%
\end{proof}

\begin{rmk} \label{VOTV1invariance}
Consider the class 
$$y_1 \otimes \frac{v_1-\gamma^{p-1} v_1}{p} + y_2 \otimes \frac{\gamma v_1-v_1}{p} + \cdots + y_{p-1} \otimes \frac{\gamma^{p-2} v_1 - \gamma^{p-3} v_1}{p}.$$
of \Cref{thm:votFormula}.  If in this formula we replace $v_1$ by $v_1' = v_1 + px$, for an arbitrary class $x \in \pi_{2p-2}^{e} R$, the resulting expression differs from the original by
$$y_1 \otimes (x-\gamma^{p-1} x) + y_2 \otimes (\gamma x-x) + \cdots + y_{p-1} \otimes (\gamma^{p-2} x - \gamma^{p-3} x).$$
This is exactly the transfer, in $\pi_{2p}^{e} (\Sigma^{1+\Yright} R)$, of $y_1 \otimes x$.  Thus, altering $v_1$ by a multiple of $p$ does not change the class $\vot$ modulo transfers.
\end{rmk}

%% file: Image_in_Ep-1.tex
In this section, we use the formula of \Cref{thm:votFormula} to compute the span of $\vot$ in the height $p-1$ theories $E_{p-1}$ and $\tmf(2)$, which we verified were $\mu_p$-orientable in \Cref{sec:Hmtpy-Purity}.
Our main result, stated in Theorems \ref{thm:tmf-section6} and \ref{thm:E-section6}, proves that the span of $\vot$ generates the homotopy of these theories in a suitable sense.
This demonstrates a height-shifting phenomenon in equivariant homotopy theory: though these theories are height $p-1$ classically, the fact that their homotopy is generated by $\vot$ indicates that they should be regarded as height $1$ objects in $C_p$-equivariant homotopy theory.

\begin{ntn}
Let $R$ denote a $C_p$-ring spectrum, equipped with a $\mu_p$-orientation
$$\Sigma^{\infty} \Cmu \to \Sigma^{1+\Yright} R.$$
Precomposition with $\vot$ then yields a map
$$S^{2\rho} \to \Sigma^{1+\Yright} R,$$
which by the dualizability of $S^{1+\Yright}$ is equivalent to a map of $C_p$-spectra
$$S^{2\rho-1-\Yright} \to R.$$
Engaging in a slight abuse of notation, we will throughout this section denote this map by
$$\vot:S^{2\rho-1-\Yright} \to R.$$
\end{ntn}

\begin{dfn}
Given a $\mu_p$-oriented $C_p$-ring $R$, applying $\pi_{2p-2}^{e}$ gives a homomorphism of $\mathbb{Z}_{(p)}[C_p]$-modules
$$\pi_{2p-2}^{e} \vot:\pi_{2p-2}^{e} S^{2\rho-1-\Yright} \to \pi_{2p-2}^{e} R.$$
\end{dfn}

The main theorems of this section are as follows:

\begin{thm} \label{thm:tmf-section6}
Suppose that
$$\Sigma^{\infty} \Cmu \to \Sigma^{1+\Yright} \tmf(2)$$
is any $\mu_3$-orientation of $\tmf(2)$.  Then the map
$\pi_{4}^{e} \vot:\pi_4^{e} S^{2\rho-1-\Yright} \to \pi_4^{e} \tmf(2)$
  is an isomorphism of $\mathbb{Z}_{(3)}$-modules, and thus also of $\mathbb{Z}_{(3)} [C_3]$-modules.
\end{thm}

\begin{thm} \label{thm:E-section6}
Suppose that
$$\Sigma^{\infty} \Cmu \to \Sigma^{1+\Yright} E_{p-1}$$
is any $\mu_p$-orientation of $E_{p-1}$.  Then the image of $\pi_{2p-2}^{e} \vot$ in $\pi_{2p-2}^{e} E_{p-1}$ maps surjectively onto the degree $2p-2$ component of $\pi_*(E_{p-1})/(p,\mathfrak{m}^2)$.
\end{thm}

\begin{rmk} \label{rmk:FpReduction}
  Note that the map $\pi_4^{e} S^{2\rho-1-\Yright} \to \pi_4^{e} \tmf(2)$ of \Cref{thm:tmf-section6} is a map of rank $2$ free $\mathbb{Z}_{(3)}$-modules.  Thus, it is an isomorphism if and only if its mod $3$ reduction is, which is a map of rank $2$ vector spaces over $\mathbb{F}_3$.

Similarly, the degree $2p-2$ component $\pi_*(E_{p-1})/(p,\mathfrak{m}^2)$ is a rank $p-1$ vector space over $\mathbb{F}_p$, generated by $u^{p-1}, u_1 u^{p-1}, u_2 u^{p-1}, \cdots, u_{p-2} u^{p-1}$.  The map $\pi^{e}_{2p-2} S^{2\rho-1-\Yright} \to \pi_{2p-2}(E_{p-1}/(p,\mathfrak{m}^2))$ of \Cref{thm:E-section6} factors through the mod $p$ reduction of its domain, after which it becomes a map of rank $p-1$ vector spaces over $\mathbb{F}_p$.

  Both Theorems \ref{thm:tmf-section6} and \ref{thm:E-section6} thus reduce to a question of whether maps of rank $p-1$ vector spaces over $\mathbb{F}_p$ are isomorphisms.  These maps are furthermore equivariant, or maps of $\mathbb{F}_p[C_p]$-modules, with the actions of $C_p$ given by reduced regular representations. We will therefore find \Cref{lem:transferlemma} below particularly useful. First, we recall some basic facts from representation theory.
\end{rmk}

\begin{rec}
  Given two $\F_p [C_p]$-modules $V$ and $W$, the space $\Hom_{\F_p} (V,W)$ inherits the structure of a $C_p$-module via conjugation, where $\gamma \in C_p$ sends $F : V \to W$ to $\gamma \circ F \circ \gamma^{-1}$.
  Then there is an identification
\[\Hom_{\F_p} (V,W)^{C_p} = \Hom_{\F_p [C_p]} (V,W),\]
so that the transfer determines a linear map
\[\Tr : \Hom_{\F_p} (V,W) \to \Hom_{\F_p [C_p]} (V,W).\]
\end{rec}

\begin{lem} \label{lem:transferlemma}
Let $\bar{\rho}$ denote the $\mathbb{F}_p[C_p]$-module corresponding to the reduced regular representation of $C_p$.  Then a homomorphism
$$\phi \in \Hom_{\mathbb{F}_p[C_p]}(\bar{\rho},\bar{\rho})$$
is an isomorphism if and only if $\phi+\Tr(\psi)$ is for any transferred homomorphism $\Tr(\psi)$.  More precisely,
$\Hom_{\mathbb{F}_p[C_p]}(\bar{\rho},\bar{\rho})$ is a local $\mathbb{F}_p[C_p]$-algebra, with maximal ideal the ideal of transferred homomorphisms.
\end{lem}

\begin{proof}
  Note that $\rhobar$ is a uniserial $\F_p [C_p]$-module, i.e. its submodules are totally ordered by inclusion. Since the endomorphism ring of a uniserial module over a Noetherian ring is local \cite[Proposition 20.20]{LamNC}, the ring $\Hom_{\F_p [C_p]} (\rhobar,\rhobar)$ is local.

  There is an identification $\rhobar ^{C_p} = \triv$, so we obtain a ring homomorphism
  \[\Hom_{\F_p [C_p]} (\rhobar,\rhobar) \to \Hom_{\F_p [C_p]} (\rhobar^{C_p}, \rhobar^{C_p}) = \Hom_{\F_p [C_p]} (\triv, \triv) = \F_p. \]
  Since this homomorphism is clearly surjective, we learn that its kernel must be equal to the maximal ideal of $\Hom_{\F_p [C_p]} (\rhobar,\rhobar)$.

  On the other hand, for any $x \in \rhobar^{C_p}$ and $\psi \in \Hom_{\F_p} (\rhobar,\rhobar)$, we have
  \[\Tr(\psi) (x) = \sum_{i = 0} ^{p-1} \gamma^i \psi( \gamma^{-i} x) = \sum_{i = 0} ^{p-1} \gamma^i \psi(x) = \Tr(\psi(x)) = 0,\]
  where the last equality follows from the fact that the transfer is zero on $\rhobar$. It follows that $\Tr(\psi)$ lies in the maximal ideal of $\Hom_{\F_p [C_p]} (\rhobar,\rhobar)$.

  Finally, the equivalence
  \[\Hom_{\F_p} (\rhobar, \rhobar) \cong \triv\{\mathrm{id}_{\rhobar}\} \oplus \mathrm{free},\]
  shows that the maximal ideal is equal to the image of $\Tr$ for dimension reasons.
\end{proof}

\begin{proof}[Proof of \Cref{thm:tmf-section6}]
  Recall that $\pi_4^{e} \tmf(2)$ is a free $\mathbb{Z}_{(3)}$-module with basis $\lambda_1$ and $\lambda_2$.
In light of \Cref{rmk:FpReduction}, it suffices to analyze the image of $v_1^{\mu_3}$ in its mod $3$ reduction, which is a free $\mathbb{F}_3$-module generated by the reductions of $\lambda_1$ and $\lambda_2$.
By combining \Cref{lem:transferlemma} with \Cref{thm:votFormula}, it suffices to show that a basis for this rank $2$ $\mathbb{F}_3$-module is given by the mod $3$ reduction of classes 
$$\frac{v_1-\gamma^2 v_1}{3},\frac{\gamma v_1 - v_1}{3} \in \pi^{e}_{4} \tmf(2).$$
Here, $v_1 \in \pi^{e}_{4} \tmf(2)$ refers to the class of \Cref{extractclassicalv1}, which depends on the chosen $\mu_3$-orientation.  By combining \Cref{VOTV1invariance} and \Cref{tmf(2)formulae}, we may as well set $v_1$ to be $-\lambda_1-\lambda_2$.  Using the formulas of \cite[Lemma 7.3]{Vesnatmf2} (cf. \Cref{tmf(2)calc}), we calculate
$$\frac{v_1-\gamma^2 v_1}{3} \equiv -\lambda_2 \mod 3, \text{ and}$$
$$\frac{\gamma v_1 - v_1}{3} \equiv \lambda_1 \mod 3.$$
These clearly generate all of $\pi_4^{e} \tmf(2)$ modulo $3$, as desired.
\end{proof}

\begin{proof}[Proof of \Cref{thm:E-section6}]
By arguments analogous to those in the previous proof, it suffices to check that
$$\frac{v_1-\gamma^{p-1} v_1}{p}, \frac{\gamma v_1 - v_1}{p}, \cdots, \frac{\gamma^{p-2} v_1 - \gamma^{p-3} v_1}{p} \in \pi^{e}_{2p-2} E_{p-1}$$
  reduce to generators of the degree $2p-2$ component of $\pi_*(E_{p-1} )/(p,\mathfrak{m}^2)$.  By \Cref{VOTV1invariance}, we may assume that $\frac{\gamma v_1 - v_1}{p}$ in $\pi^{e}_{2p-2} E_{p-1}$ is the element $v$ defined in \Cref{sec:Ethy}.   Under this assumption, the $p-1$ classes of interest become $v$ and its translates under the $C_p$ action on $\pi^{e}_{2p-2} E_{p-1}$. As noted in \Cref{rmk:Ep-1-forward}, these span $\pi_{2p-2} ^e E_{p-1}/(p,\mathfrak{m}^2)$.
\end{proof}

%% file: Main.bbl
\begin{thebibliography}{{Wil}17b}

\bibitem[AF78]{AlmFoss}
Gert Almkvist and Robert Fossum.
\newblock Decomposition of exterior and symmetric powers of indecomposable
  {${\bf Z}/p{\bf Z}$}-modules in characteristic {$p$} and relations to
  invariants.
\newblock In {\em S\'{e}minaire d'{A}lg\`ebre {P}aul {D}ubreil, 30\`eme
  ann\'{e}e ({P}aris, 1976--1977)}, volume 641 of {\em Lecture Notes in Math.},
  pages 1--111. Springer, Berlin, 1978.

\bibitem[AM78]{ArakiI}
Sh\^{o}r\^{o} Araki and Mitutaka Murayama.
\newblock {$\tau $}-cohomology theories.
\newblock {\em Japan. J. Math. (N.S.)}, 4(2):363--416, 1978.

\bibitem[BBHS19]{InvertibleEmod}
Agnes Beaudry, Irina Bobkova, Michael Hill, and Vesna Stojanoska.
\newblock {Invertible $K(2)$-Local $E$-Modules in $C_4$-Spectra}.
\newblock 2019.
\newblock \href{https://arxiv.org/abs/1901.02109}{arXiv:1901.02109}.

\bibitem[BC20]{HoodBhatt}
Prasit Bhattacharya and Hood Chatham.
\newblock On the $\mathrm{EO}$-orientability of vector bundles.
\newblock 2020.
\newblock \href{https://arxiv.org/abs/2003.03795}{arXiv:2003.03795}.

\bibitem[BHSZ20]{LTModels}
Agnes Beaudry, Michael Hill, Xiao{L}in~Danny Shi, and Mingcong Zeng.
\newblock Models of {L}ubin-{T}ate spectra via {R}eal bordism theory.
\newblock 2020.
\newblock \href{https://arxiv.org/abs/2001.08295}{arXiv:2001.08295}.

\bibitem[GH04]{GHObst}
P.~G. Goerss and M.~J. Hopkins.
\newblock Moduli spaces of commutative ring spectra.
\newblock In {\em Structured ring spectra}, volume 315 of {\em London Math.
  Soc. Lecture Note Ser.}, pages 151--200. Cambridge Univ. Press, Cambridge,
  2004.

\bibitem[GM00]{GorMah}
V.~Gorbounov and M.~Mahowald.
\newblock Formal completion of the {J}acobians of plane curves and higher real
  {$K$}-theories.
\newblock {\em J. Pure Appl. Algebra}, 145(3):293--308, 2000.

\bibitem[GM17]{RealGorenstein}
J.~P.~C. Greenlees and Lennart Meier.
\newblock Gorenstein duality for real spectra.
\newblock {\em Algebr. Geom. Topol.}, 17(6):3547--3619, 2017.

\bibitem[HH16]{HillHopkins}
Michael Hill and Michael Hopkins.
\newblock Equivariant symmetric monoidal structures.
\newblock 2016.
\newblock \href{https://arxiv.org/abs/1610.03114}{arXiv:1610.03114}.

\bibitem[HH18]{RWSI}
Michael Hill and Michael Hopkins.
\newblock Real {W}ilson {S}paces {I}.
\newblock 2018.
\newblock \href{https://arxiv.org/abs/1806.11033}{arXiv:1806.11033}.

\bibitem[HHR11]{HHRodd}
M.~A. Hill, M.~J. Hopkins, and D.~C. Ravenel.
\newblock On the 3-primary {A}rf-{K}ervaire invariant problem.
\newblock 2011.
\newblock Unpublished note available at
  \href{https://web.math.rochester.edu/people/faculty/doug/mypapers/odd.pdf}{https://web.math.rochester.edu/people/faculty/doug/mypapers/odd.pdf}.

\bibitem[HHR16]{HHR}
M.~A. Hill, M.~J. Hopkins, and D.~C. Ravenel.
\newblock On the nonexistence of elements of {K}ervaire invariant one.
\newblock {\em Ann. of Math. (2)}, 184(1):1--262, 2016.

\bibitem[Hil06]{HillThesis}
Michael~Anthony Hill.
\newblock {\em Computational methods for higher real {K}-theory with
  applications to tmf}.
\newblock ProQuest LLC, Ann Arbor, MI, 2006.
\newblock Thesis (Ph.D.)--Massachusetts Institute of Technology.

\bibitem[Hil19]{HillPurity}
Michael Hill.
\newblock Freeness and equivariant stable homotopy.
\newblock 2019.
\newblock \href{https://arxiv.org/abs/1910.00664}{arXiv:1910.00664}.

\bibitem[HK01]{HuKriz}
Po~Hu and Igor Kriz.
\newblock Real-oriented homotopy theory and an analogue of the
  {A}dams-{N}ovikov spectral sequence.
\newblock {\em Topology}, 40(2):317--399, 2001.

\bibitem[HL16]{tmflevel}
Michael Hill and Tyler Lawson.
\newblock Topological modular forms with level structure.
\newblock {\em Invent. Math.}, 203(2):359--416, 2016.

\bibitem[HLS18]{EPic}
Drew Heard, Guchuan Li, and Xiao{L}in~Danny Shi.
\newblock {Picard groups and duality for Real Morava $E$-theories}.
\newblock 2018.
\newblock \href{https://arxiv.org/abs/1810.05439}{arXiv:1810.05439}.

\bibitem[HM17]{HillMeier}
Michael~A. Hill and Lennart Meier.
\newblock The {$C_2$}-spectrum {${\rm Tmf}_1(3)$} and its invertible modules.
\newblock {\em Algebr. Geom. Topol.}, 17(4):1953--2011, 2017.

\bibitem[HMS17]{PicEO}
Drew Heard, Akhil Mathew, and Vesna Stojanoska.
\newblock Picard groups of higher real {$K$}-theory spectra at height {$p-1$}.
\newblock {\em Compos. Math.}, 153(9):1820--1854, 2017.

\bibitem[HS20]{RealOrEn}
Jeremy Hahn and XiaoLin~Danny Shi.
\newblock Real orientations of {L}ubin-{T}ate spectra.
\newblock {\em Invent. Math.}, 221(3):731--776, 2020.

\bibitem[HSWX19]{C4slice}
Michael~A. Hill, Xiao{L}in~Danny Shi, Guozhen Wang, and Zhouli Xu.
\newblock {The slice spectral sequence of a {$C_4$}-equivariant height-{$4$}
  Lubin-Tate theory}.
\newblock 2019.
\newblock \href{https://arxiv.org/abs/1811.07960}{arXiv:1811.07960}.

\bibitem[HY18]{HY}
Michael~A. Hill and Carolyn Yarnall.
\newblock A new formulation of the equivariant slice filtration with
  applications to {$C_p$}-slices.
\newblock {\em Proc. Amer. Math. Soc.}, 146(8):3605--3614, 2018.

\bibitem[KLW17]{KLWLandweberFlat}
Nitu Kitchloo, Vitaly Lorman, and W.~Stephen Wilson.
\newblock Landweber flat real pairs and {$ER(n)$}-cohomology.
\newblock {\em Adv. Math.}, 322:60--82, 2017.

\bibitem[Lam01]{LamNC}
T.~Y. Lam.
\newblock {\em A first course in noncommutative rings}, volume 131 of {\em
  Graduate Texts in Mathematics}.
\newblock Springer-Verlag, New York, second edition, 2001.

\bibitem[LLQ20]{RealTate}
Guchuan Li, Vitaly Lorman, and J.D. Quigley.
\newblock {Tate blueshift and vanishing for {R}eal oriented cohomology}.
\newblock 2020.
\newblock \href{https://arxiv.org/abs/1910.06191}{arXiv:1910.06191}.

\bibitem[LSWX19]{HurewiczImage}
Guchuan Li, XiaoLin~Danny Shi, Guozhen Wang, and Zhouli Xu.
\newblock Hurewicz images of real bordism theory and real {J}ohnson-{W}ilson
  theories.
\newblock {\em Adv. Math.}, 342:67--115, 2019.

\bibitem[Lur18]{ECII}
Jacob Lurie.
\newblock Elliptic {C}ohomology {II}: {O}rientations.
\newblock 2018.
\newblock \newline {A}vailable at
  \href{http://www.math.ias.edu/~lurie/}{http://www.math.ias.edu/~lurie/}.

\bibitem[Mil60]{Milnor}
J.~Milnor.
\newblock On the cobordism ring {$\Omega ^{\ast} $} and a complex analogue.
  {I}.
\newblock {\em Amer. J. Math.}, 82:505--521, 1960.

\bibitem[Mor89]{FormsK}
Jack Morava.
\newblock Forms of {$K$}-theory.
\newblock {\em Math. Z.}, 201(3):401--428, 1989.

\bibitem[Mos68]{Mosher}
Robert~E. Mosher.
\newblock Some stable homotopy of complex projective space.
\newblock {\em Topology}, 7:179--193, 1968.

\bibitem[MSZ20]{NormF2}
Lennart Meier, Xiao{L}in~Danny Shi, and Mingcong Zeng.
\newblock Norms of {E}ilenberg-{M}ac{L}ane spectra and {R}eal {B}ordism.
\newblock 2020.
\newblock \href{https://arxiv.org/abs/2008.04963}{arXiv:2008.04963}.

\bibitem[Nav10]{NaveEO}
Lee~S. Nave.
\newblock The {S}mith-{T}oda complex {$V((p+1)/2)$} does not exist.
\newblock {\em Ann. of Math. (2)}, 171(1):491--509, 2010.

\bibitem[PRS19]{ConjSpaces}
Wolfgang Pitsch, Nicolas Ricka, and Jerome Scherer.
\newblock {Conjugation Spaces are Cohomologically Pure}.
\newblock 2019.
\newblock \href{https://arxiv.org/abs/1908.03088}{arXiv:1908.03088}.

\bibitem[Qui69]{Quillen}
Daniel Quillen.
\newblock On the formal group laws of unoriented and complex cobordism theory.
\newblock {\em Bull. Amer. Math. Soc.}, 75:1293--1298, 1969.

\bibitem[Rav78]{Rav78}
Douglas~C. Ravenel.
\newblock The non-existence of odd primary {A}rf invariant elements in stable
  homotopy.
\newblock {\em Math. Proc. Cambridge Philos. Soc.}, 83(3):429--443, 1978.

\bibitem[Ser67]{SerreLCF}
J.-P. Serre.
\newblock Local class field theory.
\newblock In {\em Algebraic {N}umber {T}heory ({P}roc. {I}nstructional {C}onf.,
  {B}righton, 1965)}, pages 128--161. Thompson, Washington, D.C., 1967.

\bibitem[Sil09]{Silverman}
Joseph~H. Silverman.
\newblock {\em The arithmetic of elliptic curves}, volume 106 of {\em Graduate
  Texts in Mathematics}.
\newblock Springer, Dordrecht, second edition, 2009.

\bibitem[Sto12]{Vesnatmf2}
Vesna Stojanoska.
\newblock Duality for topological modular forms.
\newblock {\em Doc. Math.}, 17:271--311, 2012.

\bibitem[Wil17a]{DWThesis}
Dylan Wilson.
\newblock Equivariant, parametrized, and chromatic homotopy theory.
\newblock 2017.
\newblock Thesis (Ph.D.)--Northwestern University.

\bibitem[{Wil}17b]{DWSlices}
Dylan {Wilson}.
\newblock {On categories of slices}.
\newblock 2017.
\newblock \href{https://arxiv.org/abs/1711.03472}{arXiv:1711.03472}.

\end{thebibliography}
